\documentclass[12pt]{article}
\usepackage{amsfonts}
\usepackage{booktabs}
\usepackage{array}
\usepackage{amsthm}
\usepackage{amsmath,amssymb}
\usepackage{mathtools}
\usepackage{graphicx,floatrow,subfig}
\usepackage{multirow,makecell}
\usepackage{enumerate}
\usepackage{bm}
\usepackage{color}
\usepackage{xcolor}
\usepackage{threeparttable}
\allowdisplaybreaks[4]
\usepackage{hyperref}
\usepackage{comment}

\newtheorem{lem}{Lemma}
\newtheorem{thm}{Theorem}

\newtheorem{rem}{Remark}

\newtheorem{defi}{Definition}

\newtheorem{con}{Conjecture}
\newcommand{\sgn}{{\mathrm{sgn}}}

\allowdisplaybreaks[4]
\title{\bf The number of limit cycles of \\ piecewise linear Li\'{e}nard systems}
\author{
  Hebai Chen$^{1}$\thanks{Email: chen\_hebai@csu.edu.cn},
  Zhijie Li$^{1}$\thanks{Email: li\_zhijie@csu.edu.cn},
  Rui Zhang$^{1}$\thanks{Email: zhang\_rui@csu.edu.cn},
  Xiang Zhang$^{2}$\thanks{Email: xzhang@sjtu.edu.cn}
  \\
  {\footnotesize
    $^1$School of Mathematics and Statistics, HNP-LAMA, Central South University, Changsha 410083, China
  }\\
  {\footnotesize
    $^2$School of Mathematical Sciences, 
    Shanghai Jiao Tong University,   Shanghai 200240, China}\\
}
\date{}
\begin{document}
	\maketitle
    \hrule
\begin{abstract}
For 
the planar Li\'{e}nard differential system  $\dot{x}=F(x)-y$, $\dot{y}=x$, where $F(x)$ is a piecewise linear function, Tonnelier (SIAM J. Appl. Math., 2002) conjectured that the maximum number of limit cycles of the system is $n$ when  $F(x)$ has $n$ fold points and no jump points, and $2n$ when $F(x)$ has $n$ jump points and no fold points.
This conjecture was confirmed by Llibre et al. (J. Nonlinear Sci., 2015)  (resp. Chen et al. (J. London Math. Soc., 2026a)) when  $F(x)$ has one fold point and no jump points  (resp. two fold points). More recently, Chen et al. (J. London Math. Soc., 2026b) proved that the conjecture is correct  when  $F(x)$ has no fold points and one jump point.  All other cases remain open.

Here we verify that
the lower bound for the maximum number of limit cycles of the system can be $n$ when $F(x)$ has only $n$ fold points, and $2n$ when $F(x)$ has only $n$ jump points, thereby confirming the lower bound part of Tonnelier's conjecture.
Moreover, 
when $F(x)$ has $m$ jump points and $n-m$ fold points, $0\le m\le n$, we also show that the system can have $n+m=(n-m)+2m$ limit cycles.
In addition,  a complete classification of the {dynamics} near infinity for this class of systems is provided.


\end{abstract}

\vspace{2em}
\hrule

\section{Introduction}
Hilbert's 16th problem, proposed by Hilbert in 1900 \cite{DH}, remains one of the central open problems in the qualitative theory of planar dynamical systems. It consists of two distinct parts. Part~I  {concerns the topology of algebraic curves and surfaces}, particularly the number and arrangement of their components. Part~II addresses the maximum number and relative positions of limit cycles of planar polynomial vector fields of a given degree.
 Both parts remain major open problems in mathematics (see, for instance, \cite{Ar1,Ar2}).
 Since linear systems admit no limit cycle, the simplest nontrivial polynomial case is that of quadratic differential systems. However, the maximum number of limit cycles remains unresolved even in this simple case (see, for instance, \cite{Ba, Chen, S}). More results concerning  the second part of Hilbert's 16th problem can be found in the literature by Christopher et al. \cite{CLT} and Ilyashenko \cite{Il}.

Following Smale's insight \cite{Smale1,Smale2}, a more feasible approach is to focus on the polynomial Li\'{e}nard {differential} system
\begin{equation}
    \begin{aligned}
        \dot{x}=F(x)-y,\qquad
        \dot{y}=x,
    \end{aligned}
    \label{ls}
\end{equation}
where $F(x)$ is a polynomial in $x$. Concerning the number of limit cycles of system \eqref{ls},  {Lins, de Melo and Pugh} in 1977 posed the next conjecture.

\begin{con}[Lins--de Melo--Pugh \cite{LMP}]
    The number of limit cycles of system \eqref{ls}, with a polynomial $F(x)$  of degree $n$, is at most $\left[\frac{n-1}{2}\right]$.
\end{con}

Lins et al. \cite{LMP} proved that system \eqref{ls} has no limit cycle for $n=2$ and established that system \eqref{ls} admits at most one limit cycle for $n=3$.
The first significant progress on this conjecture was made thirty years
later by Dumortier et al.~\cite{DPR}, who obtained one more limit cycle
than the conjectured bound for odd \(n\geq7\) by using geometric
singular perturbation theory. After this work, further progress was made; see, for instance, \cite{MD} for $n=6$ and \cite{DeMH} for $n\ge 6$. Specifically, for $n\ge 6$, De Maesschalck and Huzak \cite{DeMH} in 2015 constructed examples of system \eqref{ls} in the slow-fast setting which can have $n-2$ limit cycles.
Li and Llibre \cite{LL} in 2012 confirmed the conjecture when $n=4$, that is, the maximum number of limit cycles for system \eqref{ls} is exactly one.
For $n=5$, Li and Lu \cite{LLu} proved that  the cyclicity of non-degenerate slow-fast cycles of the polynomial Li\'enard differential system \eqref{ls} in the slow-fast setting is at most $2$. That is, at most two limit cycles can bifurcate from non-degenerate slow--fast cycles through the slow divergence integral. Regardless, Lins--de Melo--Pugh's conjecture is incorrect for degree $n \ge 6$. Thus, the maximum number of limit cycles of system~\eqref{ls}
remains unknown for $n\geq5$.
A comprehensive survey {on this conjecture} can be found in the work of Llibre and Zhang \cite{LZ:c}.

	 If system~\eqref{ls} is subjected to a small perturbation of the form $
	 F(x)=\varepsilon G(x)$ with $0<\varepsilon\ll1$,
	 where $G(x)$ is a polynomial, then the associated Melnikov function can often be computed explicitly, yielding up to
$
	 \left[\frac{n-1}{2}\right]
$
	 limit cycles bifurcating from the periodic orbits of the underlying integrable system.
	 However, in physics and engineering applications, Li\'enard system~\eqref{ls} is often non-polynomial and instead takes the form of a piecewise linear system.
	 Piecewise linear Li\'enard systems have been widely used to model various nonlinear phenomena, including memristor oscillators~\cite{LPV,LPV1} and the FitzHugh--Nagumo system~\cite{CWXX,M1,M2}.
	 Due to the lack of smoothness, the construction and computation of the corresponding Melnikov functions become substantially more difficult.
	 Moreover, such systems can undergo bifurcations that generate or destroy limit cycles.
	 Motivated by these phenomena, Tonnelier~\cite{To} investigated system~\eqref{ls} with a piecewise linear function $F(x)$ and proposed the following conjecture.


\begin{con}[Tonnelier \cite{To}]\label{con1}
    $(i)$ System \eqref{ls}, with $F(x)$ continuous and piecewise linear on $n + 1$ intervals, has up to $n$ limit cycles. \\
    $(ii)$ System \eqref{ls}, with $F(x)$ piecewise linear on $n + 1$ intervals and having $n$ finite jump discontinuities, has up to $2n$ limit cycles.
\end{con}

The terminology for switch points in this conjecture for the piecewise linear function $F(x)$ is defined as follows.
 \begin{defi}\label{def}
     Let $F(x)$ be a piecewise linear function with $S_i:=(x_i,F(x_i))$ on a switching line \(x=x_i\). Denote
     \[
F(x_i^-)=\lim_{x\to x_i^-}F(x), \qquad
F(x_i^+)=\lim_{x\to x_i^+}F(x)
\]
and
$$
F'(x_i^-)=\lim_{h\to 0^-}\frac{F(x_i+h)-F(x_i)}{h},\qquad F'(x_i^+)=\lim_{h\to 0^+}\frac{F(x_i+h)-F(x_i)}{h}.
$$
\begin{enumerate}[$(i)$]
\item
A point $S_i$ is called a jump point of $F(x)$ if $F(x)$ is
discontinuous at $S_i$, that is,
\[
F(x_i^-)\neq F(x_i^+).
\]
\item
A point $S_i$ is called a fold point of $F(x)$ if $F(x)$ is
continuous at $S_i$ but its left derivative is not equal to the right derivative,  that is,
\[
F(x_i^-)=F(x_i^+)=F(x_i), \qquad
F'(x_i^-)\neq F'(x_i^+).
\]
\end{enumerate}
 \end{defi}

	Concerning  piecewise linear Li\'enard system \eqref{ls} with  function $F(x)$ containing only fold points,  Llibre et al. in \cite{LPV} proved that system \eqref{ls}, in which $F(x)$ has a unique  fold point and no jump points, admits at most one limit cycle. Chen et al. \cite{CFJM} further investigated the global dynamics of system~\eqref{ls}, in which $F(x)$ has two fold points and no jump points,  and proved that the maximum number of limit cycles is exactly two for system \eqref{ls}.  Therefore, part~$(i)$ of Tonnelier's conjecture  has been positively resolved for $n=1$ and $n=2$. However, up to now, the conjecture remains open for $n\ge3$.
	Moreover,	Llibre et al. in  \cite{LP,LPZ} provided  some examples of system \eqref{ls} which can exhibit $n$ limit cycles  when  $F(x)$ has $2n+2$ fold points and no jump points, for arbitrary positive integer $n$.
	
For piecewise linear Li\'enard systems~\eqref{ls} in which \(F(x)\) contains only jump points, Chen et al.~\cite{CDFZ} proved that system~\eqref{ls} admits at most two limit cycles when \(F(x)\) has a single jump point and no fold points. They also constructed an example possessing exactly two limit cycles, thereby giving a positive answer to part~\((ii)\) of Tonnelier's conjecture for \(n=1\). However, the conjecture remains unresolved for $n\geq 2$.

Motivated by Tonnelier's conjecture and the existing results for systems involving only fold points or only jump points \cite{CDFZ,CFJM,CJT21,CJT,CLXY,LPV,LPV1,LPZ}, together with the evidence obtained from our constructions, we propose the following conjecture concerning the general case in which $F(x)$ contains both jump and fold points. To the best of our knowledge, this problem remains completely open.

\begin{con}
	\label{con:m}
	System \eqref{ls}, with $F(x)$ piecewise linear on $n+1$ intervals and having $m$ jump points and $n-m$ fold points, has up to $n+m$ limit cycles, where $0\le m\le n$.
\end{con}
\begin{rem}\label{rem1}
	Clearly, Conjecture~$\ref{con:m}$ is a generalization of Tonnelier's conjecture $($i.e., Conjecture~$\ref{con1})$.
	It reduces to part $(i)$ of Tonnelier's conjecture when $m=0$ and to part $(ii)$ when $m=n$.	
	\end{rem}
The remainder of this paper is organized around Conjectures  \ref{con1} and \ref{con:m}. In Section \ref{ds}, we investigate the dynamics at infinity of system \eqref{ls}. Since the behavior near infinity plays a crucial role in the global qualitative analysis of piecewise linear Li\'enard systems, we provide a complete classification of the phase portraits on the Poincar\'e disc. In Section \ref{sc}, we study the  lower bounds for the maximum number of limit cycles of system \eqref{ls}. Precisely, we establish three main theorems, which provide supporting evidence for Conjectures \ref{con1} and \ref{con:m}. Finally, Subsection~\ref{sec5} contains some concluding remarks.

\section{Dynamics of piecewise linear Li\'enard system near infinity}
\label{ds}
 The dynamics near infinity plays a fundamental role in the qualitative analysis of system \eqref{ls}, particularly in the study of limit cycles. The behavior near infinity of system~\eqref{ls} with polynomial $F(x)$
has been investigated by Dumortier and Herssens~\cite{DHF}. However, a complete analysis near infinity of piecewise linear Li\'{e}nard system \eqref{ls} is still unavailable.

Without loss of generality, we consider the piecewise linear Li\'{e}nard system \eqref{ls} with a piecewise linear function $F(x)$ defined on $n+1$ intervals in the form
\begin{equation}\label{Fx-representation}
	F(x)=ax+b+\sum_{i=1}^n a_i |x-x_i|
	+\sum_{i=1}^n b_i\,\operatorname{sgn}(x-x_i),
\end{equation}
where $a,b,a_i,b_i,x_i\in\mathbb{R}$, $(a_i^2+b_i^2)x_i\neq0$ for all $i=1,2,\cdots n$, and $x_1<x_2<\cdots<x_n$. Here, $\operatorname{sgn}(\cdot)$ denotes the sign function. We remark that the condition $a_i^2+b_i^2\neq0$ guarantees that the piecewise linear system \eqref{ls} is defined on exactly $n+1$ regions, while the condition $x_i\neq0$ ensures that the unique equilibrium point of the system does not lie on any switching line.
Then, the plane $\mathbb{R}^2$ is divided into $n+1$ open regions,
$$
\begin{aligned}
&I_1=\{(x,y)\in\mathbb R^2: x<x_1\},\\
&I_2=\{(x,y)\in\mathbb R^2: x_1<x<x_2\},\\
&\cdots,\\
&I_{n}=\{(x,y)\in\mathbb R^2: x_{n-1}<x<x_n\},\\
&I_{n+1}=\{(x,y)\in\mathbb R^2: x>x_n\}
\end{aligned}
$$
by the $n$ switching lines
$$
L_i=\{(x,y)\in\mathbb R^2: x=x_i\},\qquad  i=1,2,\dots,n.
$$
It is worth mentioning that, for a given piecewise linear Li\'{e}nard system \eqref{ls} with $n+1$ regions, the representation of $F(x)$ is not unique; however, it provides a unified description of fold and jump points.
By Definition \ref{def}, the point $(x_i,F(x_i))$ is a fold point if and only if $b_i=0$ and $a_i\neq0$, whereas it is a jump point if and only if $b_i\neq0$.

On the leftmost region $I_1$ and the rightmost region $I_{n+1}$, the function $F(x)$ reduces to
\begin{equation}\label{Fl}
F(x)= \left(a-\sum_{i=1}^n a_i\right) x+b+\sum_{i=1}^n \left(a_ix_i-b_i\right):=a_l x+b_l
\end{equation}
and
\begin{equation}\label{Fr}
F(x)= \left(a+\sum_{i=1}^n a_i\right) x+b+\sum_{i=1}^n \left(-a_ix_i+b_i\right):=a_r x+b_r,
\end{equation}
respectively. Define
\begin{equation}\label{K1}
k_1=\frac{a_r-\sqrt{a_r^2-4}}{2},\quad
k_2=\frac{a_r+\sqrt{a_r^2-4}}{2},\quad
k_3=\frac{a_l+\sqrt{a_l^2-4}}{2},\quad
k_4=\frac{a_l-\sqrt{a_l^2-4}}{2}.
\end{equation}
In this paper, we show that, when $F(x)$ is piecewise linear,  except for the degenerate case $a_l+a_r=0$, the dynamics near infinity of system \eqref{ls} is completely determined  on the Poincar\'e disc by the parameters $a_l$ and $a_r$.
 We stress that this result only describes the dynamics near infinity. In the intermediate regions,
 more complicated dynamical behaviors may arise  due to the piecewise structure of $F(x)$.
 For simplicity, we restrict our attention to the case  $a_r\ge0$. The case $a_r<0$ can be reduced to the above case via the transformation $(x,y,a_r)\mapsto(-x,-y,-a_r)$, which yields analogous conclusions.
 We classify the parameter space into $15$ cases
\[
{\bf(H):}~\begin{array}{lll}
{\bf(1a) } ~a_r > 2,\ a_l > 2; &
{\bf(1b) }~ a_r > 2,\ a_l = 2; &
{\bf(1c) } ~a_r > 2,\ -2 < a_l < 2; \\
{\bf(1d) }~ a_r > 2,\ a_l = -2; &
{\bf(1e) }~ a_r > 2,\ a_l < -2; &
{\bf(2a) }~ a_r = 2,\ a_l > 2; \\
{\bf(2b) }~ a_r = 2,\ a_l = 2; &
{\bf(2c) }~ a_r = 2,\ -2 < a_l < 2; &
{\bf(2d) }~ a_r = 2,\ a_l = -2; \\
{\bf(2e) }~ a_r = 2,\ a_l < -2; &
{\bf(3a) }~ 0 \leq a_r < 2,\ a_l > 2; &
{\bf(3b) }~ 0 \leq a_r < 2,\ a_l = 2; \\
{\bf(3c) }~ 0 \leq a_r < 2,\ -2 < a_l < 2; &
{\bf(3d) }~ 0 \leq a_r < 2,\ a_l = -2; &
{\bf(3e) }~ 0 \leq a_r < 2,\ a_l < -2.
\end{array}
\]

Before stating our main results, we introduce the following notions.
	To study the phase portraits of system \eqref{ls} near infinity, we employ the Poincar\'{e} compactification. This tool is described in Chapter 5 of the book by Dumortier et al. \cite{DLA}. Roughly speaking, this compactification identifies $\mathbb{R}^2$ with the interior of the closed unit disc centered at the origin of $\mathbb{R}^2$ and extends the differential system analytically to its boundary, which is usually called the equator $\mathbb S^1$  of the Poincar\'e sphere, corresponding to the infinity of $\mathbb R^2$. The unit disc is in fact the projection of the Poincar\'e sphere to the plane containing the equator. The Poincar\'{e} compactification can be extended to piecewise linear systems. For instance, we refer the reader to \cite{LiL}.

For system \eqref{ls}, assume that the equator $\mathbb S^1$ is either a closed orbit or a heteroclinic loop formed by finitely many equilibria at infinity on $\mathbb S^1$, each of which possesses hyperbolic sectors separated by $\mathbb S^1$. Then, the infinity of system \eqref{ls} is called a {\it center} (resp. {\it focus}) if all orbits in an inner neighborhood of $\mathbb S^1$ are periodic (resp. spiral around $\mathbb S^1$ without being closed).  Furthermore, the infinity is a {\it stable} (resp. an {\it unstable}) {\it focus} if $\mathbb S^1$ is the $\omega$-limit set (resp. $\alpha$-limit set) of  nearby orbits.  We remark that the notions of a center and a focus at infinity, together with the
stability of a focus, are standard terminology when the infinity is treated as an  {equilibrium} in the Bendixson compactification.

For a planar autonomous differential system, let $\varphi_t$ denote its associated flow, and let $\Gamma \subset \mathbb{R}^2$ be an orbit of the system.  
Denote by
\[
\Gamma^+(\mathbf{x})=\{\varphi_t(\mathbf{x}): t\ge0\} \quad \mbox{\rm and} \quad
\Gamma^-(\mathbf{x})=\{\varphi_t(\mathbf{x}): t\le0\}
\]
the positive and negative semi--orbits with the initial point $\mathbf{x}\in \mathbb R^2$, respectively.
The orbit $\Gamma$ is said to be {\it positively bounded} (resp.
{\it negatively bounded}) if there exists a bounded set
$\mathcal{M}\subset\mathbb{R}^2$ such that
$\Gamma^+(\mathbf{x})\subset\mathcal{M}$ (resp.
$\Gamma^-(\mathbf{x})\subset\mathcal{M}$). The orbit $\Gamma$ is called {\it bounded} if $\Gamma$ is both positively and negatively bounded.

\begin{thm}\label{thm0}
 Assume that $F(x)$ is a piecewise linear function of the form \eqref{Fx-representation}.
Then,
the dynamics of system~\eqref{ls}  near infinity
is completely determined by the outermost slopes $a_l$ and $a_r$ except in the case where  $F(x)\equiv F(-x)$ for $x>\max\{|x_1|,|x_n|\}$ but $F(x)\not\equiv F(-x)$ for  $x>0$.
More precisely, the following statements hold.
\begin{enumerate}[{\bf (i)}]
\item
In  the cases {\bf (1a)}--{\bf (1c)}, {\bf (2a)}--{\bf (2c)} and {\bf (3a)}--{\bf (3b)},
the locations and topological structures of equilibria at infinity of system \eqref{ls} are shown in {\rm Table \ref{table}} and  the corresponding phase portraits near infinity are shown in
{\rm Figures~\ref{dia:glo}(a-c), (f-h)} and {\rm (k-l)}, respectively.

\item
In the cases {\bf (3d)}--{\bf (3e)},
the locations and topological structures of equilibria at infinity of system \eqref{ls} are shown in {\rm Table \ref{table}} and the corresponding phase portraits near infinity are shown in
{\rm Figures~\ref{dia:glo}(n-o)}, respectively.

\item
In the cases {\bf (1d)}--{\bf (1e)} and {\bf (2d)}--{\bf (2e)},
the locations and topological structures of equilibria at infinity of system \eqref{ls} are shown in {\rm Table \ref{table}} and the corresponding phase portraits near infinity are shown in
{\rm Figures~\ref{dia:glo}(d-e)} and {\rm (i-j)}, respectively.

\item
In the case {\bf (3c)}, system \eqref{ls} has no equilibria at infinity, and the corresponding phase portraits near infinity are shown in {\rm Figure~\ref{dia:glo}(m)}.
Furthermore, the infinity is a center, as shown in
Figure {\rm \ref{dia:cen}(c)}, when $F(x)\equiv F(-x)$ for  $x>0$. When $F(x)\not \equiv F(-x)$ for  $x>0$,
  the {infinity}  is a focus. Moreover, any orbit sufficiently close to infinity is positively $($resp. negatively$)$ bounded if $a_l+a_r<0$, or $a_l+a_r=0$ and $b_r<b_l$  $($resp.  $a_l+a_r>0$, or $a_l+a_r=0$ and $b_r>b_l$$)$, as shown in {\rm Figure \ref{dia:cen}(a)} $($resp. {\rm Figure \ref{dia:cen}(b)}$)$.
    \end{enumerate}
\end{thm}

\begin{table}[htbp!]
\centering
\footnotesize
\setlength{\tabcolsep}{2pt}
\begin{tabular}{|p{4.7cm}|p{1.7cm}|p{4.1cm}|p{5.3cm}|}
\hline
\textbf{Parameters} & \textbf{Number} & \textbf{Location} & \textbf{Topological structure} \\ \hline
 \multirow{4}{*}{$(a_r,a_l)$ in {\bf (1a)}} & \multirow{4}{*}{4} & $I_B^+: y = k_1 x, x \to +\infty$ &   stable node, see Figure \ref{dia:glo}(a) \\ \cline{3-4}
 &  & $I_C^+: y = k_2 x, x \to +\infty$ &   saddle, see Figure \ref{dia:glo}(a)  \\ \cline{3-4}
 &  & $I_F^-: y = k_3 x, x \to -\infty$ &   stable node,  see Figure \ref{dia:glo}(a)  \\ \cline{3-4}
 &  & $I_G^-: y = k_4 x, x \to -\infty$ &   saddle, see Figure \ref{dia:glo}(a) \\ \hline
 \multirow{3}{*}{$(a_r,a_l)$ in {\bf (1b)}} & \multirow{3}{*}{3} & $I_B^+: y = k_1 x, x \to +\infty$ &   stable node, see Figure \ref{dia:glo}(b) \\ \cline{3-4}
 &  & $I_C^+: y = k_2 x, x \to +\infty$ &   saddle, see Figure \ref{dia:glo}(b) \\ \cline{3-4}
 &  & $I_A^-: y = x, x \to -\infty$ &   saddle--node, see Figure \ref{dia:glo}(b) \\ \hline
 \multirow{2}{*}{$(a_r,a_l)$ in {\bf (1c)}} & \multirow{2}{*}{2} & $I_B^+: y = k_1 x, x \to +\infty$ & stable node, see Figure \ref{dia:glo}(c) \\ \cline{3-4}
 &  & $I_C^+: y = k_2 x, x \to +\infty$ &   saddle, see Figure \ref{dia:glo}(c) \\ \hline

\multirow{3}{*}{$(a_r,a_l)$ in {\bf (1d)}} & \multirow{3}{*}{3} & $I_B^+: y = k_1 x, x \to +\infty$ &   stable node, see Figure \ref{dia:glo}(d) \\ \cline{3-4}
 &  & $I_C^+: y = k_2 x, x \to +\infty$ &   saddle, see Figure \ref{dia:glo}(d) \\ \cline{3-4}
 &  & $I_E^-: y = -x, x \to -\infty$ &   saddle--node, see Figure \ref{dia:glo}(d) \\ \hline
 \multirow{4}{*}{$(a_r,a_l)$ in {\bf (1e)}} & \multirow{4}{*}{4} & $I_B^+: y = k_1 x, x \to +\infty$ &  stable node, see Figure \ref{dia:glo}(e) \\ \cline{3-4}
 &  & $I_C^+: y = k_2 x, x \to +\infty$ &   saddle, see Figure \ref{dia:glo}(e) \\ \cline{3-4}
 &  & $I_F^-: y = k_3 x, x \to -\infty$ &   unstable node, see Figure \ref{dia:glo}(e) \\ \cline{3-4}
 &  & $I_G^-: y = k_4 x, x \to -\infty$ &   saddle, see Figure \ref{dia:glo}(e) \\\hline

   \multirow{3}{*}{$(a_r,a_l)$ in {\bf (2a)}} & \multirow{3}{*}{3} & $I_A^+: y = x, x \to +\infty$ &  saddle--node, see Figure \ref{dia:glo}(f) \\ \cline{3-4}
 &  & $I_F^-: y = k_3 x, x \to -\infty$ &   stable node, see Figure \ref{dia:glo}(f) \\ \cline{3-4}
 &  & $I_G^-: y = k_4x, x \to -\infty$ &   saddle, see Figure \ref{dia:glo}(f) \\ \hline

 \multirow{2}{*}{$(a_r,a_l)$ in {\bf (2b)}} & \multirow{2}{*}{2} & $I_A^+: y = x, x \to +\infty$ &   saddle--node, see Figure \ref{dia:glo}(g) \\ \cline{3-4}
 &  & $I_A^-: y = x, x \to -\infty$ &   saddle--node, see Figure \ref{dia:glo}(g) \\ \hline

$(a_r,a_l)$ in {\bf (2c)} & 1 & $I_A^+: y = x, x \to +\infty$ &   saddle--node, see Figure \ref{dia:glo}(h) \\ \hline

 \multirow{2}{*}{$(a_r,a_l)$ in {\bf (2d)}} & \multirow{2}{*}{2} & $I_A^+: y = x, x \to +\infty$ &   saddle--node, see Figure \ref{dia:glo}(i) \\ \cline{3-4}
 &  & $I_E^-: y = -x, x \to -\infty$ &   saddle--node, see Figure \ref{dia:glo}(i) \\ \hline

 \multirow{3}{*}{$(a_r,a_l)$ in {\bf (2e)}} & \multirow{3}{*}{3} & $I_A^+: y = x, x \to +\infty$ &   saddle--node, see Figure \ref{dia:glo}(j) \\ \cline{3-4}
 &  & $I_F^-: y = k_3 x, x \to -\infty$ &   unstable node, see Figure \ref{dia:glo}(j) \\ \cline{3-4}
 &  & $I_G^-: y =k_4 x, x \to -\infty$ &   saddle, see Figure \ref{dia:glo}(j) \\ \hline

   \multirow{2}{*}{$(a_r,a_l)$ in {\bf (3a)}} & \multirow{2}{*}{2} & $I_F^-: y = k_3 x, x \to -\infty$ &   stable node, see Figure \ref{dia:glo}(k) \\ \cline{3-4}
 &  & $I_G^-: y = k_4 x, x \to -\infty$ &   saddle, see Figure \ref{dia:glo}(k) \\ \hline

$(a_r,a_l)$ in {\bf (3b)} & 1 & $I_A^-: y = x, x \to -\infty$ &   saddle--node, see Figure \ref{dia:glo}(l) \\ \hline

$(a_r,a_l)$ in {\bf (3c)} & 0 &None  & see Figure \ref{dia:glo}(m)  \\ \hline

$(a_r,a_l)$ in {\bf (3d)} & 1 & $I_E^-: y = -x, x \to -\infty$ &   saddle--node, see Figure \ref{dia:glo}(n) \\ \hline

 \multirow{2}{*}{$(a_r,a_l)$ in {\bf (3e)}} & \multirow{2}{*}{2} & $I_F^-: y = k_3 x, x \to -\infty$ &  unstable node, see Figure \ref{dia:glo}(o) \\ \cline{3-4}
 &  & $I_G^-: y = k_4 x, x \to -\infty$ &   saddle, see Figure \ref{dia:glo}(o) \\ \hline
\end{tabular}
\caption{Locations and topological structure of equilibria at infinity of system \eqref{ls} when $F(x)$ is piecewise linear.  {Here `Number' represents the number of equilibria at infinity, and `Location' denotes the locations of the equilibria, which are determined by the slopes of the orbits approaching the equilibria.}}
\label{table}
\end{table}
\begin{figure}[!htbp]
	\centering
	\subfloat[for $(a_r,a_l)$ in {\bf (1a)}]
	{\includegraphics[scale=0.45]{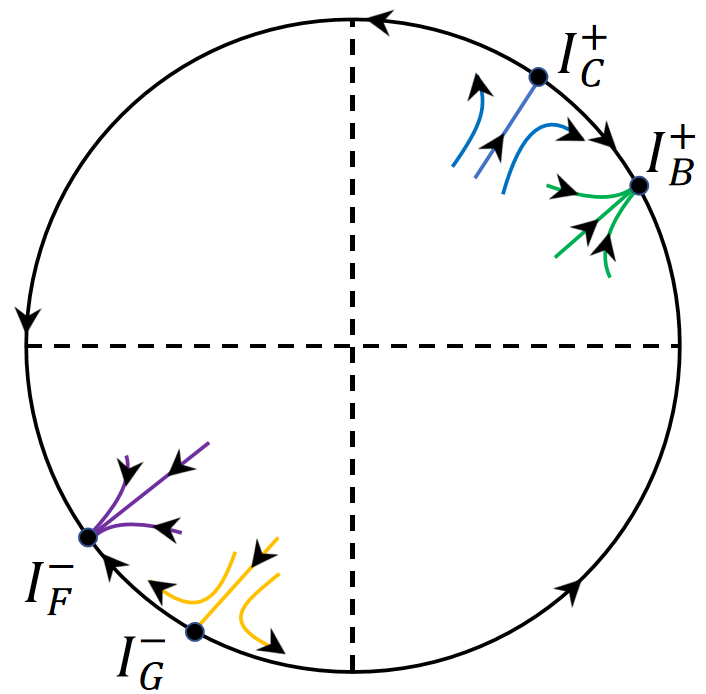}}
	\subfloat[for $(a_r,a_l)$ in {\bf (1b)}]
	{\includegraphics[scale=0.28]{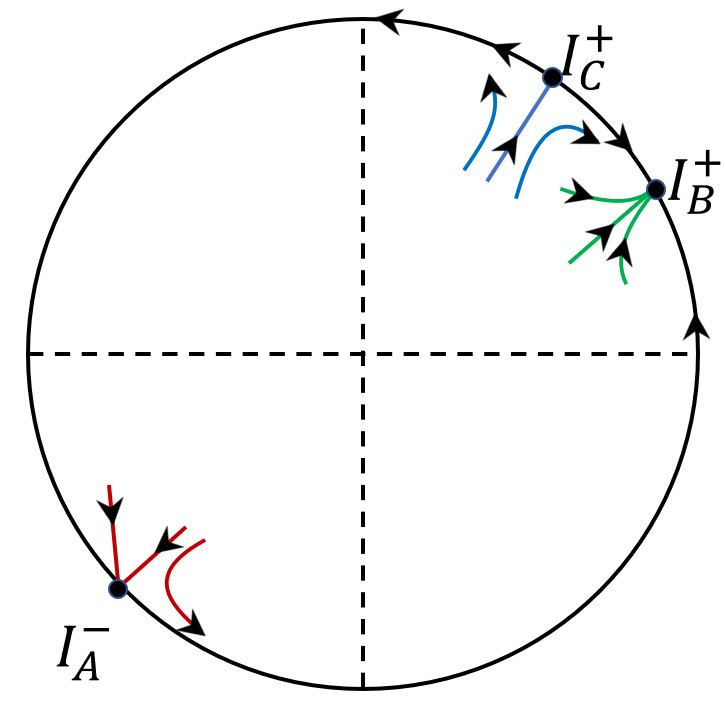}}
	\subfloat[for $(a_r,a_l)$ in {\bf (1c)}]
	{\includegraphics[scale=0.45]{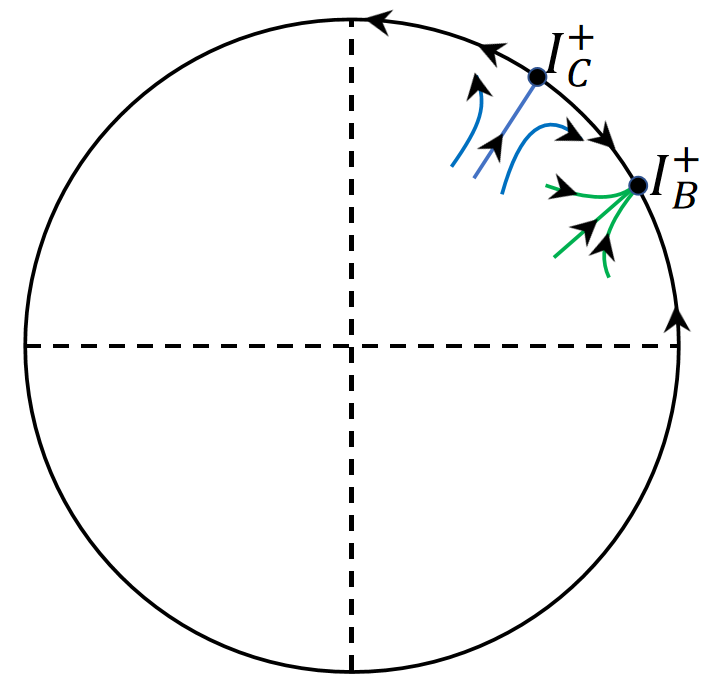}}
	\subfloat[for $(a_r,a_l)$ in {\bf (1d)}]
	{\includegraphics[scale=0.45]{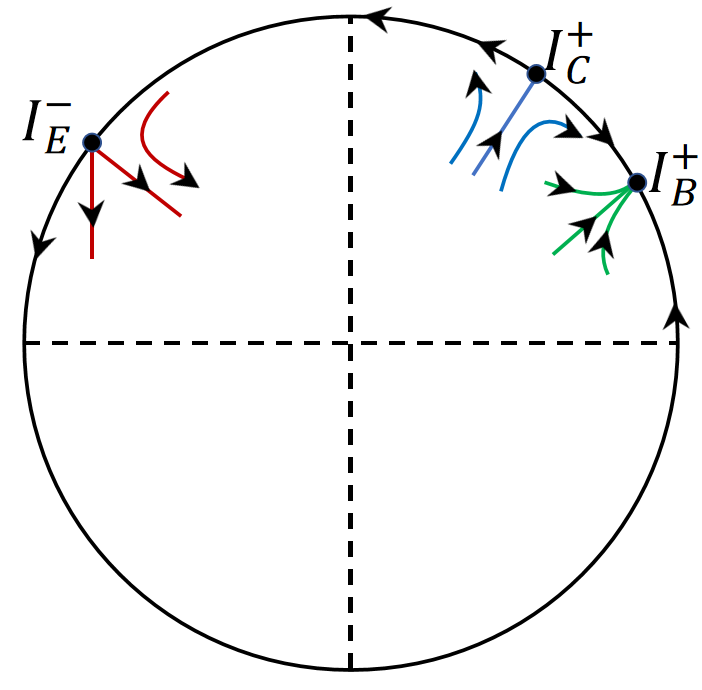}}\\[6pt]
	\subfloat[for $(a_r,a_l)$ in {\bf (1e)}]
	{\includegraphics[scale=0.45]{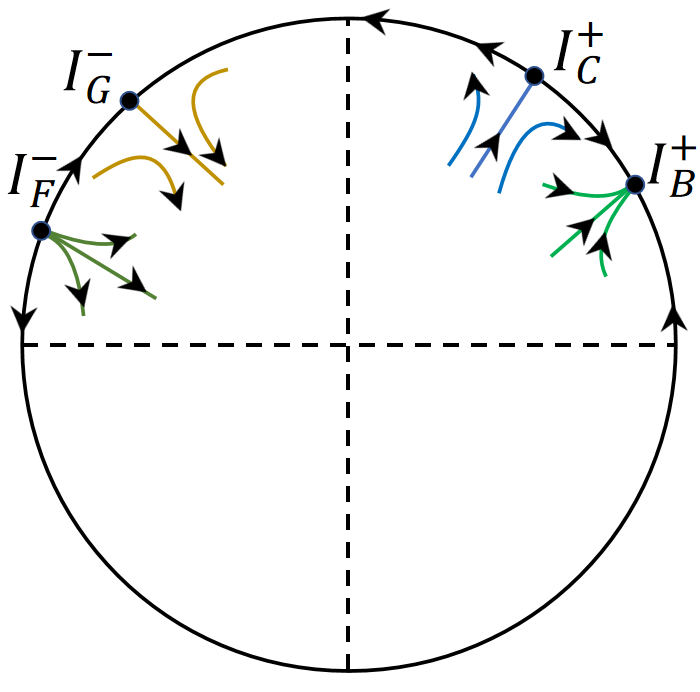}}
	\subfloat[for $(a_r,a_l)$ in {\bf (2a)}]
	{\includegraphics[scale=0.45]{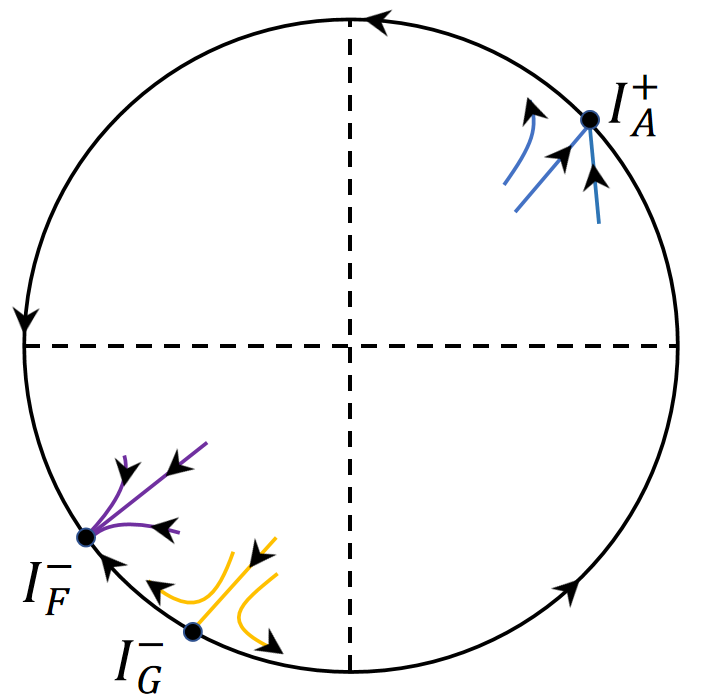}}
	\subfloat[for $(a_r,a_l)$ in {\bf (2b)}]
	{\includegraphics[scale=0.28]{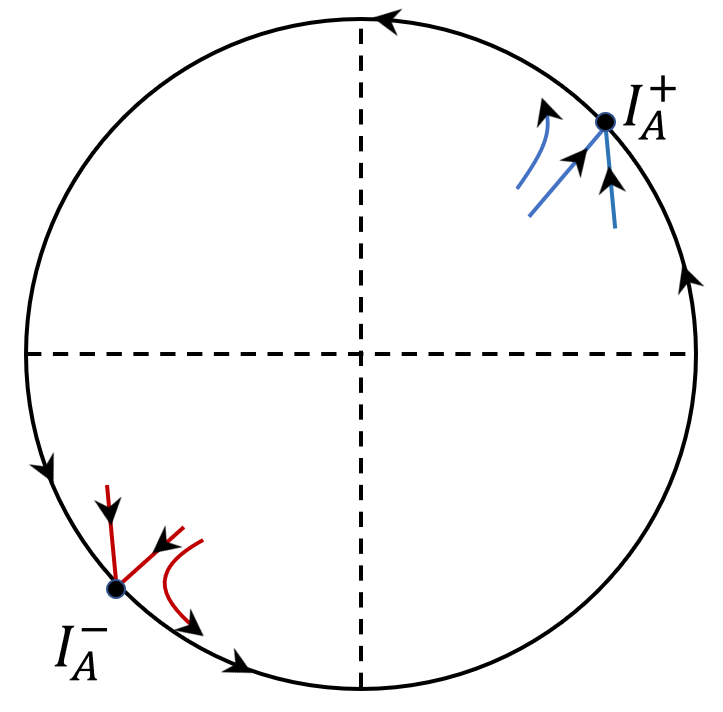}}
	\subfloat[for $(a_r,a_l)$ in {\bf (2c)}]
	{\includegraphics[scale=0.45]{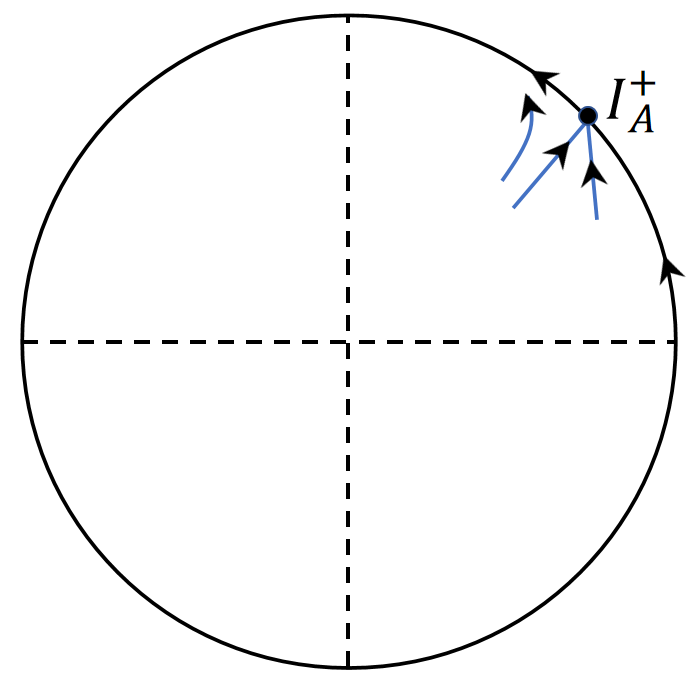}}\\[6pt]
	\subfloat[for $(a_r,a_l)$ in {\bf (2d)}]
	{\includegraphics[scale=0.45]{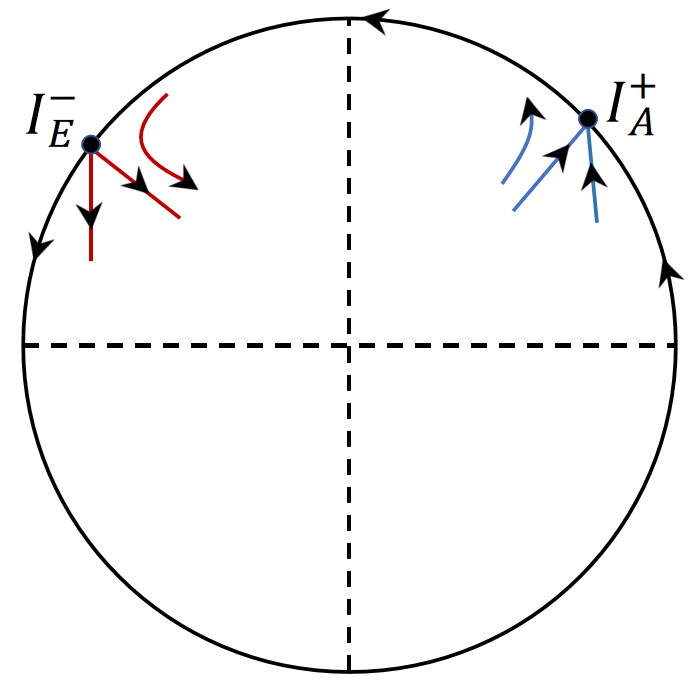}}
	\subfloat[for $(a_r,a_l)$ in {\bf (2e)}]
	{\includegraphics[scale=0.45]{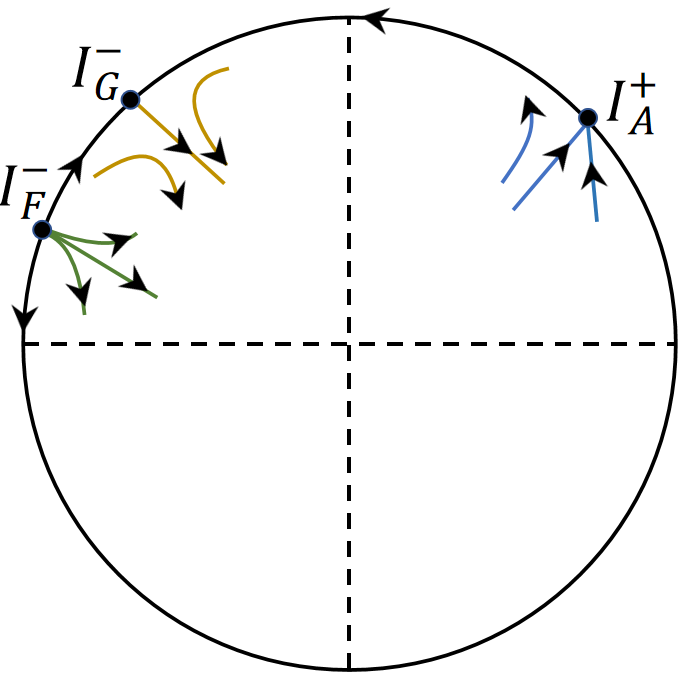}}
	\subfloat[for $(a_r,a_l)$ in {\bf (3a)}]
	{\includegraphics[scale=0.45]{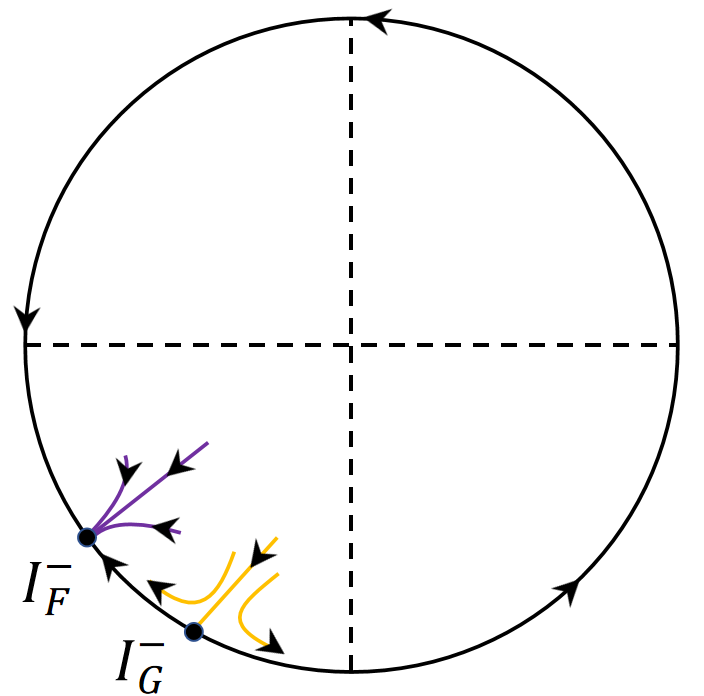}}
	\subfloat[for $(a_r,a_l)$ in {\bf (3b)}]
	{\includegraphics[scale=0.28]{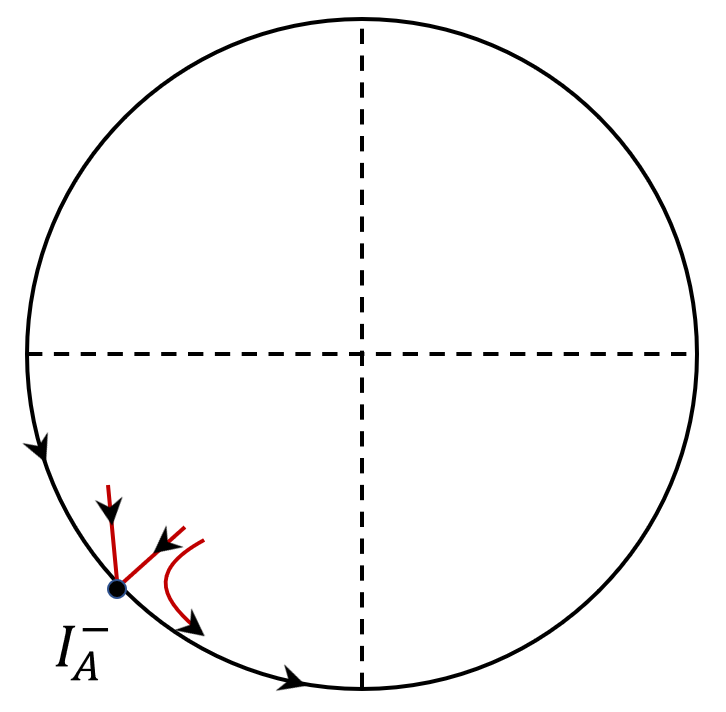}}\\[6pt]
	\subfloat[for $(a_r,a_l)$ in {\bf (3c)}]
	{\includegraphics[scale=0.45]{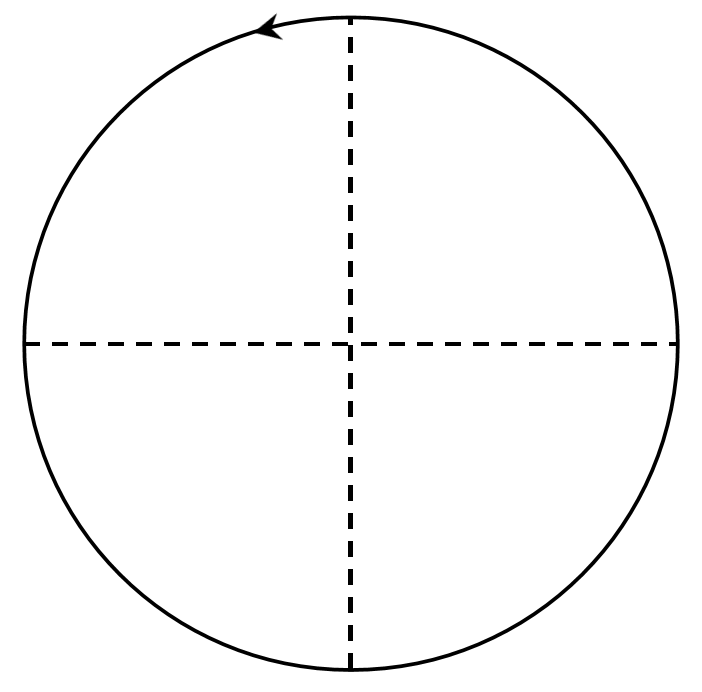}}
	\subfloat[for $(a_r,a_l)$ in {\bf (3d)}]
	{\includegraphics[scale=0.45]{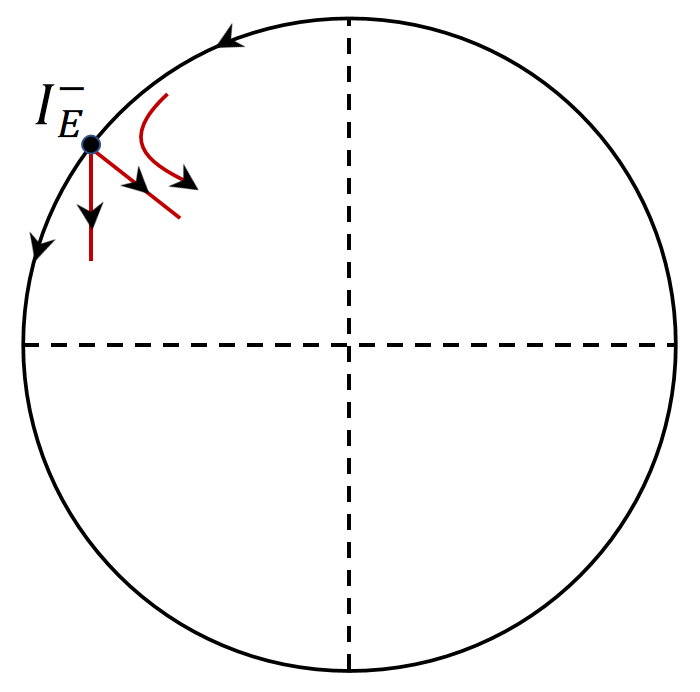}}
	\subfloat[for $(a_r,a_l)$ in {\bf (3e)}]
	{\includegraphics[scale=0.45]{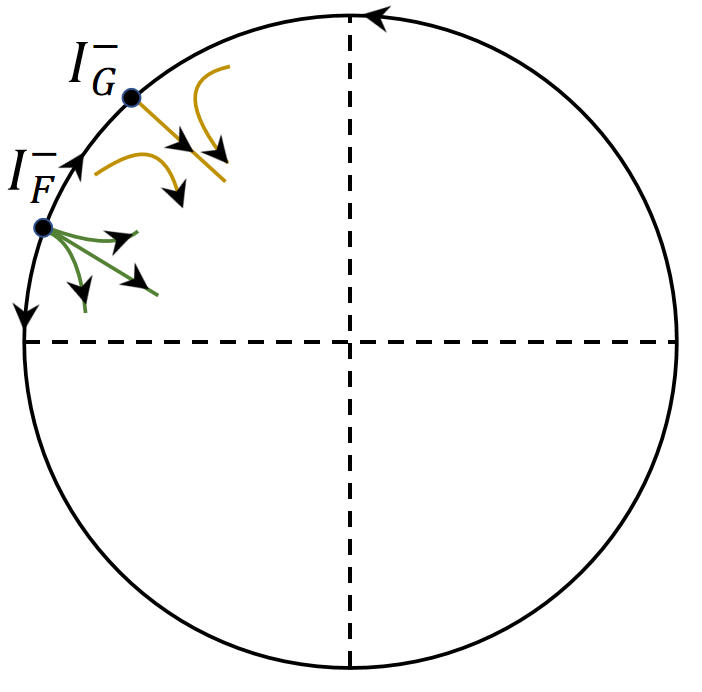}}
	\caption{Phase portraits near infinity of system \eqref{ls} on the Poincar\'{e} disc.}
	\label{dia:glo}
\end{figure}

\begin{figure}[htp]
	\centering
    	\subfloat[positively bounded]
    {\includegraphics[scale=0.25]{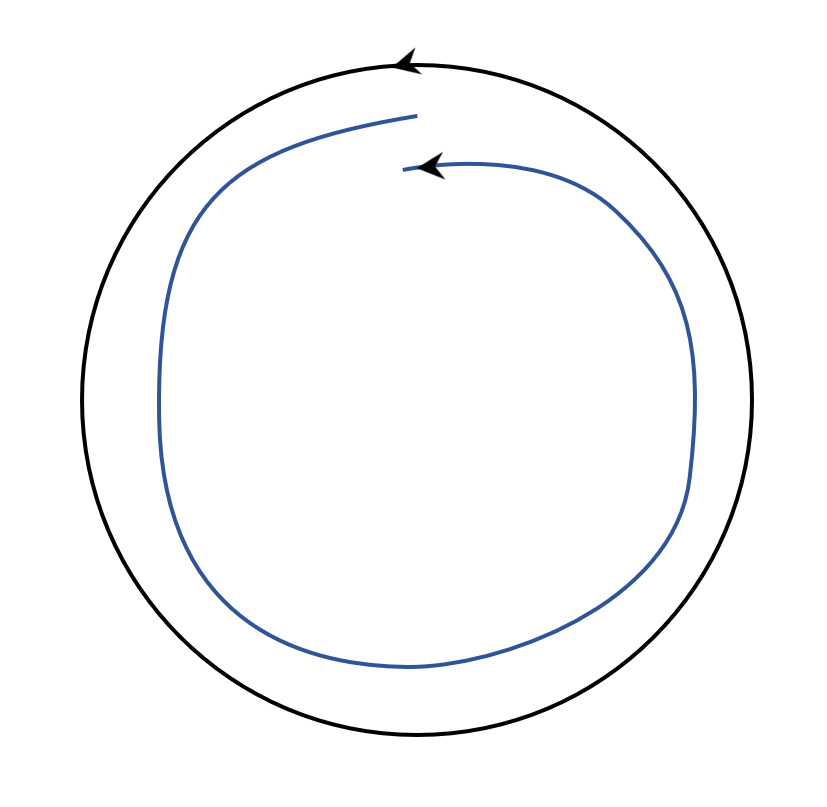}}\hspace{10pt}
    	\subfloat[negatively bounded]
	{\includegraphics[scale=0.25]{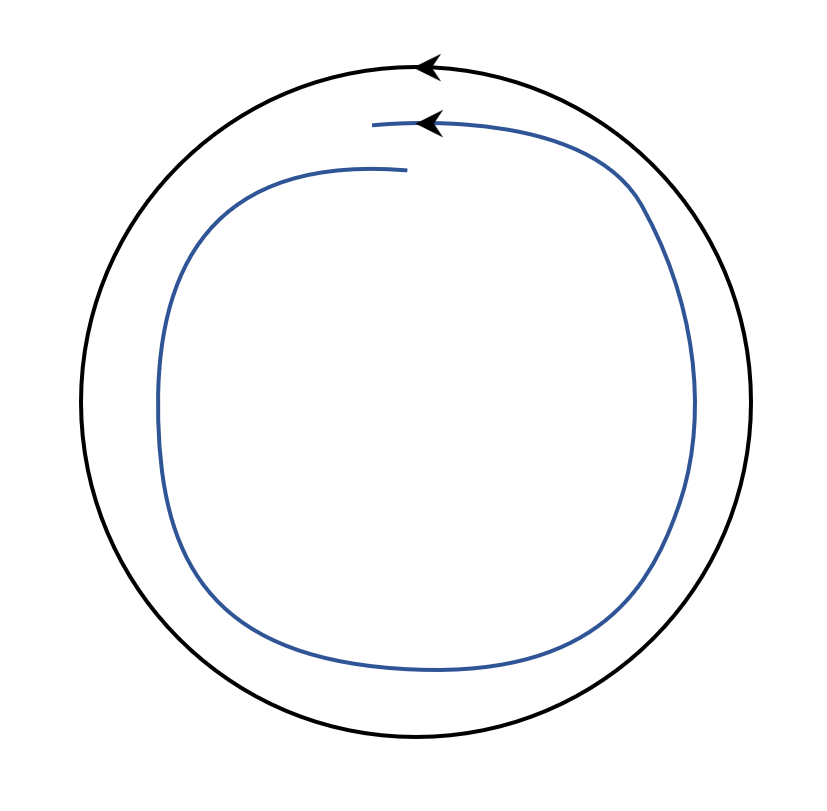}}
    	\subfloat[bounded]
	{\includegraphics[scale=0.25]{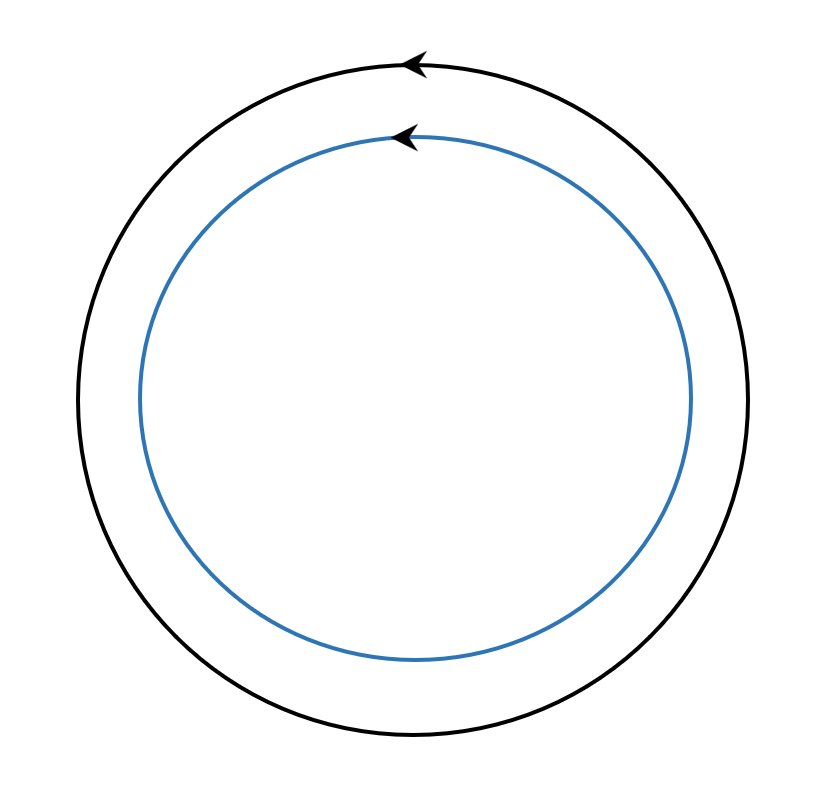}}\hspace{10pt}
	\caption{Illustration of boundedness for the orbits of system \eqref{ls} when $(a_r,a_l)$ belongs to case {\bf (3c)}.}
	\label{dia:cen}
\end{figure}

	\begin{rem}
		In the case  {\bf (3c)}, when $F(x)\equiv F(-x)$ for $x>\max\{|x_1|,|x_n|\}$ we have $a_l+a_r=0$ and $b_l=b_r$.   Then,
 the infinity may be either a center or a focus, leading to a center--focus problem at infinity. Since  $F(x)\not\equiv F(-x)$ for  $x>0$, a further distinction between these two possibilities requires additional analysis and will not be discussed in the present paper.
	\end{rem}

To prove  Theorem~\ref{thm0}, we first
study the equilibria of system \eqref{ls}  at infinity.
	\begin{lem}\label{bounded:1}
		Assume that $F(x)$ is a piecewise linear function of the form \eqref{Fx-representation}.
		Then, there are no equilibria at the endpoints of the $y$-axis
			on the Poincar\'{e} disc and all equilibria at infinity lie in the left and right half planes. Moreover, the number of equilibria at infinity on the Poincar\'{e} disc is completely determined by $a_l$ and $a_r$, as shown in the following:
		\begin{enumerate}[$(a)$]
			\item If $|a_l|<2$ $($resp. $=2$; $>2$$)$, then there are $0$ $($resp. $1$; $2$$)$ equilibria at infinity in the left half plane.
			\item If $|a_r|<2$ $($resp. $=2$; $>2$$)$, then there are $0$ $($resp. $1$; $2$$)$ equilibria at infinity in the right half plane.
		\end{enumerate}
		Furthermore, the classification of the locations and topological structures of these equilibria is summarized in {\rm Table~\ref{table}}, while the corresponding phase portraits near infinity on the Poincar\'e disc are illustrated in {\rm Figure~\ref{dia:glo}}.
	\end{lem}
	
	\begin{proof}
		To analyze the dynamical behavior of system \eqref{ls} at infinity except those on the $y$-axis, we {employ} the Poincar\'{e} transformation
		\[
		x = \frac{1}{z},\qquad y = \frac{v}{z}.
		\]
		Then system \eqref{ls} becomes \begin{equation}\label{vz}
			\dot{v} = 1 - v zF\!\left(\frac{1}{z}\right) + v^{2},\qquad \dot{z} = -z^{2}F\!\left(\frac{1}{z}\right) + z v
		\end{equation}
		with
		\begin{equation}
			zF\left(\frac{1}{z}\right)=a+bz+\sum_{i=1}^na_iz\left|\frac{1}{z}-x_i\right|+\sum_{i=1}^nb_iz\sgn\left(\frac{1}{z}-x_i\right).
			\label{eq:fz}
		\end{equation}
		To analyze the dynamical behavior of system \eqref{ls} at infinity  on the $y$-axis, we {employ} the  Poincar\'{e} transformation
		\[
		x = \frac{u}{z},\qquad y = \frac{1}{z},
		\]
	which transforms system~\eqref{ls} into
		\begin{equation}\label{uz}
			\dot{u} = -u^{2} - 1 + zF\!\left(\frac{u}{z}\right), \qquad \dot{z} = -u z + z^{2}F\!\left(\frac{u}{z}\right)
		\end{equation}
		with
		\begin{equation*}
			zF\left(\frac{u}{z}\right)=au+bz+\sum_{i=1}^na_iz\left|\frac{u}{z}-x_i\right|+\sum_{i=1}^nb_iz\sgn\left(\frac{u}{z}-x_i\right).
		\end{equation*}
		According to the method of  Poincar\'e compactification, see \cite[Chapter 5]{DLA}, we shall analyze  all equilibria of system \eqref{vz} on $z=0$ and the origin of system \eqref{uz}, respectively.
		When $z=0$, system \eqref{uz} is reduced to
		\begin{equation*}
			\dot{u} = -u^{2} - 1, \qquad \dot{z} =0,
		\end{equation*}
		which means that the origin  is not an equilibrium.  Thus, the equilibria at infinity of system \eqref{vz} can only be those outside the $y$-axis, and we only need to consider system \eqref{vz}.

		From \eqref{eq:fz}, for $|x|$ sufficiently large, we have
		$$
		zF\!\left(\frac{1}{z}\right)=
		\begin{cases}
			a_l + b_l z, & z\to0^-,\\[6pt]
			a_r + b_r z, & z\to0^+,
		\end{cases}
		$$
		where $a_l$, $a_r$, $b_l$, $b_r$ are defined by \eqref{Fl} and \eqref{Fr}.  Hence, system \eqref{vz} is reduced to
		\begin{equation}\label{z+}
			\dot{v} = 1-a_r v-b_r vz+v^2, \qquad \dot{z} =-a_r z-b_r z^2+zv,   \qquad \text { as } z\to 0^+
		\end{equation}
		and
		\begin{equation}\notag
			\dot{v} = 1-a_l v-b_l vz+v^2, \qquad \dot{z} =-a_l z-b_l z^2+zv,   \qquad \text { as } z\to 0^-.
		\end{equation}
		When \(z=0\), the number of equilibria of system \eqref{z+} is determined by the number of real roots of the equation $
		v^2-a_rv+1=0$.
		Specifically, system \eqref{z+} has no equilibria for \(0\leq a_r<2\), a unique equilibrium \(A(1,0)\) for \(a_r=2\), and two equilibria \(B(k_1,0)\) and \(C(k_2,0)\) for \(a_r>2\), where $k_1$ and $k_2$ are defined in \eqref{K1}. Moreover, it is straightforward to verify that \(B\) is a stable node and \(C\) is a saddle according to the Jacobian matrix of system \eqref{z+} at $B$ and $C$, respectively. As for $A$,  a transformation $v\mapsto v+1$ yields that
		$$
		\dot v=-b_rz-b_r vz+v^2,\qquad \dot z=-z-b_r z^2+zv.
		$$
		By \cite[Theorem 2.19 of Chapter 2]{DLA}, the origin of the system is a saddle--node, so is $A$ of system \eqref{z+}, see {\rm Figure~\ref{dia:sd}}.
		\begin{figure}[htp]
			\centering
			\includegraphics[width=0.5\linewidth]{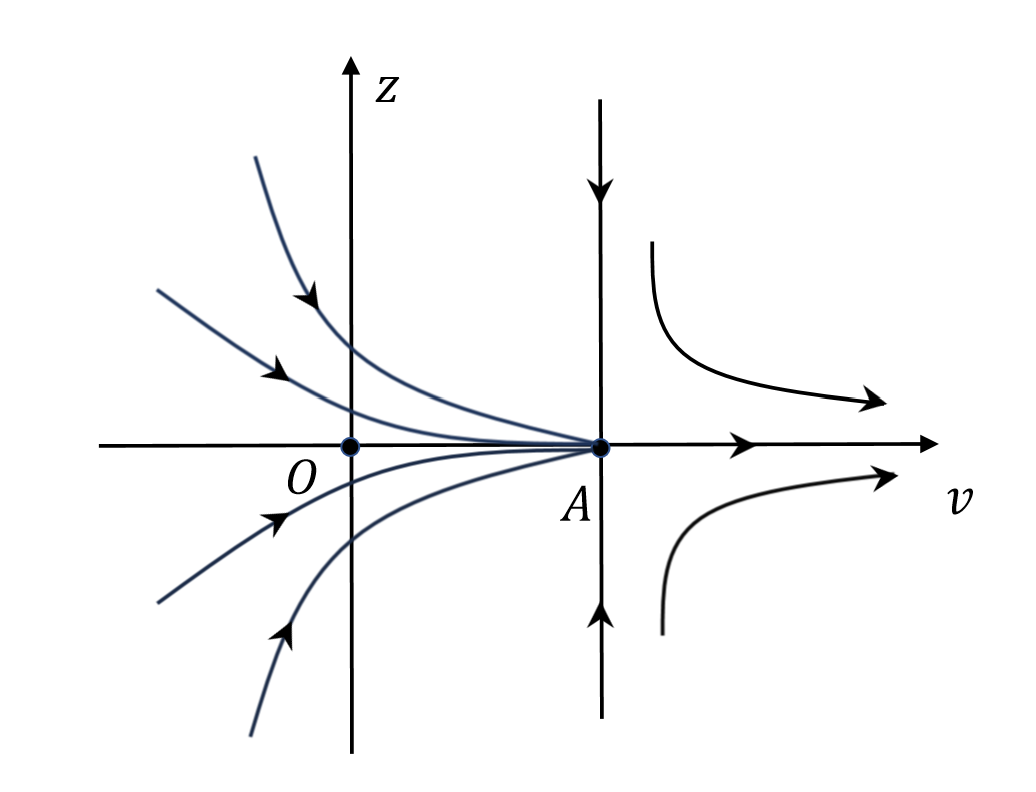}
			\caption{The qualitative properties of \(A\) for system \eqref{z+}}
			\label{dia:sd}
		\end{figure}
		Let \(I_A^+\), \(I_B^+\) and \(I_C^+\) denote the corresponding equilibria at infinity associated with \(A\), \(B\) and \(C\), respectively, in the \((v,z)\)-plane.  Therefore, there are \(0\), \(1\) (\(I_A^+\)), or \(2\) (\(I_B^+\) and \(I_C^+\)) equilibria at infinity in the right half plane for \(0\le a_r<2\), \(a_r=2\), or \(a_r>2\), respectively.
		
		A completely analogous argument for \(z\to 0^{-}\) shows that the number of equilibria at infinity in the left half-plane is \(0\), \(1\) (\(I_E^-\) or \(I_A^-\)), or \(2\) (\(I_F^-\) and \(I_G^-\)) when \(|a_l|<2\), \(|a_l|=2\), or \(|a_l|>2\), respectively. Their locations and topological structures can be obtained analogously and are summarized in {\rm Table~\ref{table}}. The corresponding phase portraits near infinity on the Poincar\'e disc are illustrated in {\rm Figure~\ref{dia:glo}}.
	\end{proof}
	By Lemma \ref{bounded:1},
	in the Poincar\'e compactification,
	only the outermost linear segments
	$F(x)=a_lx+b_l$ and $F(x)=a_rx+b_r$
	contribute to the leading-order terms of the vector field
	{near} the equator $z=0$.
	Consequently, the dynamical behavior in a neighborhood of infinity is completely determined by the outermost slopes $a_l$ and $a_r$ of $F(x)$.

	When \(0\leq a_r<2\) and \(|a_l|<2\), namely, case {\bf (3c)}, system \eqref{ls} has no equilibria at infinity on the Poincar\'e disc. Hence, the infinity is either a center or a focus. The following result characterizes this dichotomy and describes the boundedness of the orbits near infinity.
	\begin{lem}
		\label{bounded}
		Assume that $F(x)$ is a piecewise linear function of the form \eqref{Fx-representation}.
		For $(a_r,a_l)$ in case {\bf (3c)}, excluding the specific scenario  where  $F(x)\equiv F(-x)$ for $x>\max\{|x_1|,|x_n|\}$ but $F(x)\not\equiv F(-x)$ for  $x>0$, the following statements hold:
		\begin{enumerate}[$(a)$]
			\item If $F(x)\equiv F(-x)$ for all $x>0$, any orbit sufficiently close to infinity of system~\eqref{ls} is bounded, as shown in {\rm Figure \ref{dia:cen}(c)}.
			
			\item If $a_l+a_r<0$, any orbit  sufficiently close to infinity of system~\eqref{ls} is positively bounded, as shown in {\rm Figure \ref{dia:cen}(a)}.
			
			\item If $a_l+a_r>0$, any orbit  sufficiently close to infinity of system~\eqref{ls} is negatively bounded, as shown in {\rm Figure \ref{dia:cen}(b)}.
			
			\item If $a_l+a_r=0$,
			{any orbit}  sufficiently close to infinity of system~\eqref{ls} is positively $($resp.  negatively$)$ bounded for $b_r<b_l$ $($resp. $b_r>b_l$$)$, as shown in {\rm Figure \ref{dia:cen}(a)} $($resp. {\rm Figure \ref{dia:cen}(b)}$)$.
		\end{enumerate}
	\end{lem}
	\begin{proof}
		Let \(\widehat{ABC}\) denote an orbit arc of system~\eqref{ls} sufficiently close to infinity, where \(A(0,y_A)\) and \(C(0,y_C)\) lie on the positive \(y\)-axis, \(B(0,y_B)\) lies on the negative \(y\)-axis, and \(|y_A|\), \(|y_B|\) and \(|y_C|\) are sufficiently large, see  Figure \ref{fig:1}.
		\begin{figure}[htp]
			\centering
			\includegraphics[width=0.4\linewidth]{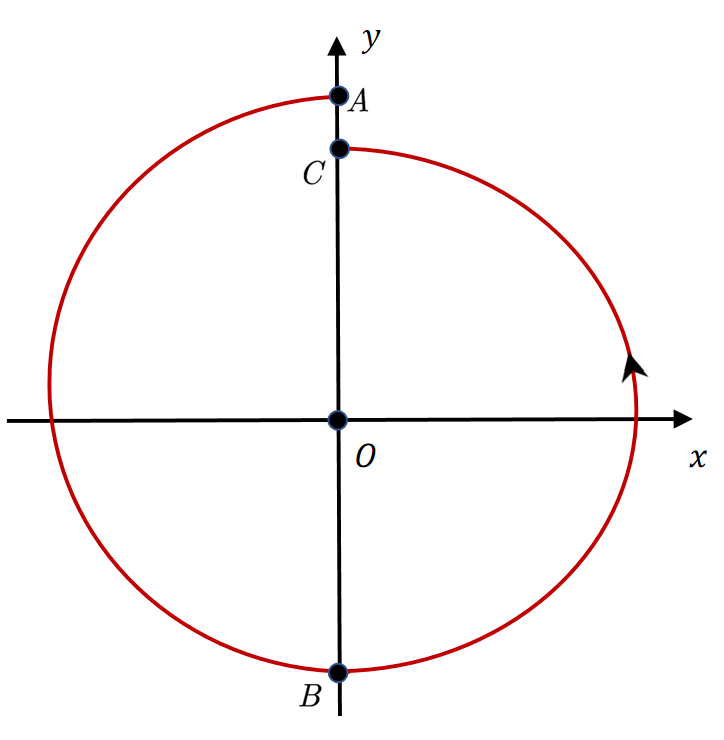}
			\caption{An orbit arc  \(\widehat{ABC}\) of system~\eqref{ls} sufficiently close to infinity.}
			\label{fig:1}
		\end{figure}
		Then,
		\(\widehat{ABC}\) is
		either positively bounded if $y_A>y_C$, { or} negatively bounded if $y_A<y_C$, or bounded (equivalently, periodic) if $y_A=y_C$.  Since \(\widehat{ABC}\) is sufficiently close to infinity, all intersection points of \(\widehat{ABC}\) with the lines \(L_i\) lie on \(L_i^c\) for each switching line. 	
		
		If $F(x)\equiv F(-x)$ for  $x>0$, the  vector field of system \eqref{ls} is symmetric with respect to the $y$-axis. Thus, \(\widehat{ABC}\)  is closed and bounded and we finish the proof of statement $(a)$.

Next, we consider the case where $F(x)\not\equiv F(-x)$.
		For $x>0$, {write system~\eqref{ls} in an equivalent way as the equation}
		\begin{equation}
			\frac{dx}{dy}=\frac{F(x)-y}{x}.
			\label{eq:pos}
		\end{equation}
		For $x<0$, applying the transformation $(x,t)\mapsto(-x,-t)$ to system~\eqref{ls},
		we obtain
		\begin{equation}
			\frac{dx}{dy}=\frac{F(-x)-y}{x}.
			\label{eq:neg}
		\end{equation}
			Assume that \(a_l+a_r<0\). Letting $x_0=\max\left\{|x_1|,|x_n|,(b_l-b_r)/(a_l+a_r)\right\}$, we have that
		$$
		F(-x)=-a_lx+b_l,\qquad F(x)=a_r x+b_r
		$$
		and
		$F(-x)>F(x)$  when $x>x_0$. Let \(\widehat{BHLC}\) be the integral curve of equation \eqref{eq:pos} connecting \(B\) and \(C\), where \(H(x_H,F(x_H))\) and \(L(x_L,F(-x_L))\) lie on the curves \(y=F(x)\) and \(y=F(-x)\), respectively. Since \(\widehat{BHLC}\) is sufficiently close to infinity, it intersects the line \(x=x_0\) exactly twice. Denote the intersection points by \(M(x_0,y_M)\) and \(N\) (see the red curve in Figure \ref{fig:abc}).
		\begin{figure}[htp]
			\centering
			\includegraphics[width=0.5\linewidth]{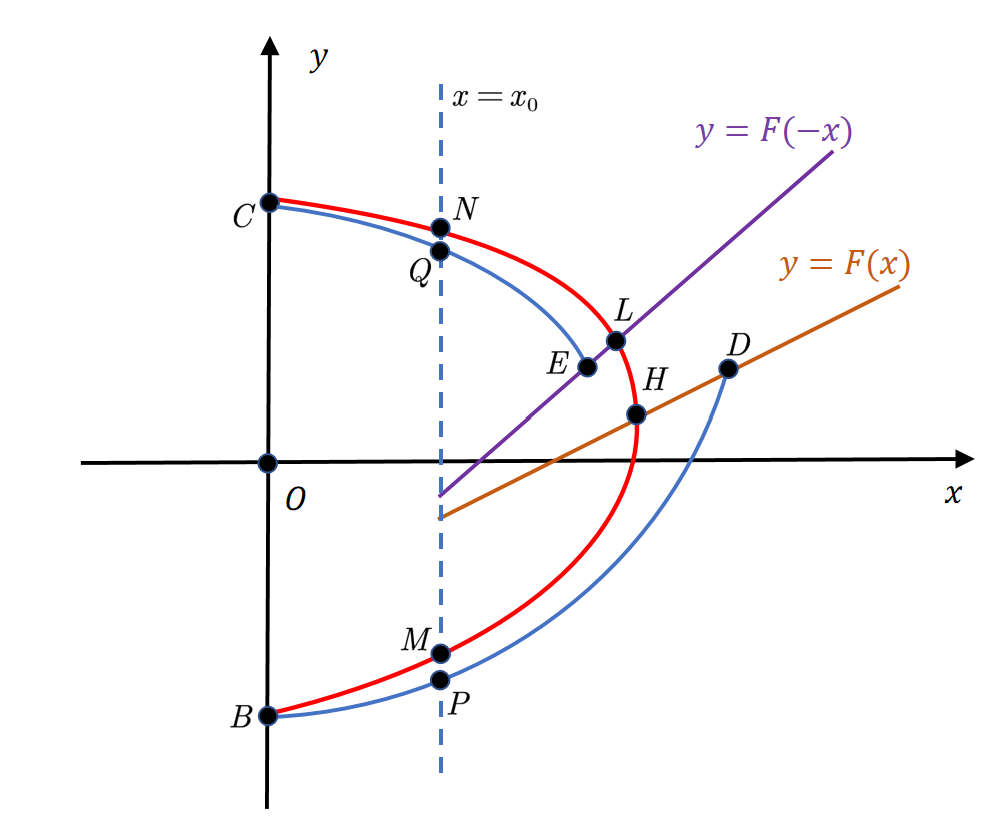}
			\caption{Integral curves  of equations \eqref{eq:pos} and \eqref{eq:neg} in $(x,y)$-plane.}
			\label{fig:abc}
		\end{figure}

		Consider the integral curve \(\widehat{BPD}\) of equation \eqref{eq:neg}, where \(P(x_0,y_P)\) lies on the line \(x=x_0\), and \(D(x_D,F(x_D))\) lies on the curve \(y=F(x)\).  We claim that \(x_D>x_H\). We now prove this claim.
		On the one hand, it follows from \eqref{eq:pos} and \eqref{eq:neg} that
		$$
		y_M-y_B=\int_0^{x_0}\frac{x}{F(x)-y_1(x)}dx,\qquad
		y_P-y_B=\int_0^{x_0}\frac{x}{F(-x)-y_2(x)}dx,
		$$
		where \(y_2(x)\) (resp. \(y_1(x)\)) denotes the integral curve \(\widehat{BPD}\) (resp. \(\widehat{BMH}\)).
		Note that \(y_2(x)\) (resp. \(y_1(x)\)) is well defined, since \(\widehat{BPD}\) (resp. \(\widehat{BMH}\)) lies below the curve \(y=F(x)\), and intersects the lines \(x=-x_i\) for all \(x_i<0\), \(i=1,2,\ldots,n\) (resp. \(x=x_j\) for all \(x_j>0\), \(j=1,2,\ldots,n\)) at points belonging to the crossing sets.
		Given that $|y_B|$ is sufficiently large, for any sufficiently small \(\varepsilon>0\), one has
		$$y_{M}-y_{B}<\varepsilon/2,\qquad
		y_{P}-y_{B}<\varepsilon/2.$$
		Consequently, we obtain $$|y_{M}-y_{P}|<\varepsilon.$$
				On the other hand,
		let $z(x) = y_1(x) - y_2(x)$. Consequently, we have $|z(x_0)| = |y_{M} - y_{P}| < \varepsilon$. 
		Since \(F(-x)>F(x)\) for \(x>x_0\), and since the two curves are considered
		before their first intersections with curve \(y=F(x)\). In particular,
		\(y_1(x)<F(x)\) and \(y_2(x)<F(x)<F(-x)\). We have
		\[
		F(x)-y_1(x)>0,\qquad F(-x)-y_2(x)>0.
		\]
		Therefore,
		\[
		\begin{aligned}
			z'(x)
			&=
			\frac{x}{F(x)-y_1(x)}
			-
			\frac{x}{F(-x)-y_2(x)}
			\\
			&=
			\frac{x\bigl(F(-x)-F(x)+y_1(x)-y_2(x)\bigr)}
			{\bigl(F(x)-y_1(x)\bigr)\bigl(F(-x)-y_2(x)\bigr)}
			\\
			&=
			A(x)+B(x)z(x),
		\end{aligned}
		\]
		where
		\[
		A(x)=
		\frac{x\bigl(F(-x)-F(x)\bigr)}
		{\bigl(F(x)-y_1(x)\bigr)\bigl(F(-x)-y_2(x)\bigr)}>0
		\]
		and
		\[
		B(x)=
		\frac{x}
		{\bigl(F(x)-y_1(x)\bigr)\bigl(F(-x)-y_2(x)\bigr)}.
		\]
		By the variation-of-constants formula, for any
	\[
	x_*\in \left(x_0,\bar x\right), \qquad
	\bar x=\min\{x_H,x_D\},
	\]
		we have
		\[
		z(x_*)
		=
		z(x_0)\exp\left(\int_{x_0}^{x_*}B(s)\,ds\right)
		+
		\int_{x_0}^{x_*}
		A(s)
		\exp\left(\int_s^{x_*}B(\tau)\,d\tau\right)\,ds .
		\]
		The second term is strictly positive. Denote it by
		\[
		I(x_*)
		=
		\int_{x_0}^{x_*}
		A(s)
		\exp\left(\int_s^{x_*}B(\tau)\,d\tau\right)\,ds>0.
		\]
		Since the orb arc is chosen sufficiently close to infinity, the quantity
		\(|z(x_0)|=|y_M-y_P|\) can be made arbitrarily small. Hence we may choose it
		so that
		\[
		|z(x_0)|
		<
		\frac{I(x_*)}
		{2\exp\left(\int_{x_0}^{x_*}B(s)\,ds\right)}.
		\]
		Consequently,
		\(
		z(x_*)>0.
		\)

It follows from \eqref{eq:pos} and \eqref{eq:neg} that
		\begin{equation}
			\label{12}
			\begin{aligned}
				z(x)&=\int_{x_0}^{x}\frac{\xi}{F(\xi)-y_1(\xi)}d\xi-\int_{x_0}^{x}\frac{\xi}{F(-\xi)-y_2(\xi)}d\xi	+z(x_0)\\
				&=z(x_0)+\int_{x_0}^{x} \frac{\xi(F(-\xi)-F(\xi))}{(F(\xi)-y_1(\xi))(F(-\xi)-y_2(\xi))}d\xi\\
				&\qquad+		
				\int_{x_0}^{x} \frac{\xi\cdot z(\xi)}{(F(\xi)-y_1(\xi))(F(-\xi)-y_2(\xi))}d\xi\\
				=:&\varphi(x)+\int_{x_0}^{x}\phi(\xi)z(\xi)d\xi,
			\end{aligned}
		\end{equation}
		where
		$$
		\varphi(x)=z(x_0)+\int_{x_0}^{x} \frac{\xi(F(-\xi)-F(\xi))}{(F(\xi)-y_1(\xi))(F(-\xi)-y_2(\xi))}d\xi
		$$
		and
		$$
		\phi(x)=\frac{x}{(F(x)-y_1(x))(F(-x)-y_2(x))}.
		$$
		Solving  the integration \eqref{12}, we obtain the following result:
		\begin{equation}\label{23}
			\begin{aligned}
				z(x)&=\varphi(x)+\int_{x_0}^x\phi(\xi)\varphi(\xi)\exp\left(\int_\xi^x\phi(s)ds\right)d\xi,\\
				&=z(x_0)\exp\left(\int_{x_0}^{x}\phi(\xi)d\xi\right)+\int_{x_0}^{x}\varphi'(\xi)\exp\left(\int_{\xi}^{x}\phi(s)ds\right)d\xi.
			\end{aligned}
		\end{equation}
		Since
		$$
		\varphi'(x)=\frac{x(F(-x)-F(x))}{(F(x)-y_1(x))(F(-x)-y_2(x))}>0
		$$
		for $x\in [x_0,\bar x)$,
		it follows  that
		$$
		\int_{x_0}^{x}\varphi'(\xi)\exp\left(\int_{\xi}^{x}\phi(s)ds\right)d\xi>0.
		$$
Moreover, \(|z(x_0)|<\varepsilon\). More precisely, the initial point
\(B\) can be chosen sufficiently close to infinity so that the first term
in \eqref{23} is dominated by the positive integral term. Therefore, for \(x_*\) sufficiently close to \(\bar x\) from the left, one
obtains \(z(x_*)>0\). If \(x_D\le x_H\), then \(\bar x=x_D\). Since
\(y_2(x_D)=F(x_D)\), whereas \(y_1(x_D)<F(x_D)\), we would have
\(z(x_D)<0\), a contradiction. Hence \(x_D>x_H\).
		
		Let \(\widehat{CQE}\) be the integral curve of equation \eqref{eq:neg} starting from \(C\), where \(Q\) lies on the line \(x=x_0\), and \(E(x_E,F(-x_E))\) lies on the curve \(y=F(-x)\). Similarly, one obtains $
		x_E<x_L$.
		The relative locations are illustrated in Figure~\ref{fig:abc}.
		Finally, by extending the integral curve \(\widehat{BD}\) of equation \eqref{eq:neg} to \(A\), it follows immediately that $
		y_A>y_C$. In other words,   \(\widehat{ABC}\) is
		positively bounded. This completes the proof of statement $(b)$.
		
		The remaining statements $(c)$ and $(d)$ can be proved by arguments analogous to those used for the statement $(b)$. The details are omitted.
	\end{proof}
	
\begin{rem}
	Here the smallness of \(z(x_0)\) is used in the following sense. For fixed
	\(x_*\in(x_0,\bar x)\), the integral term in \eqref{23} is strictly positive
	because \(\varphi'(x)>0\). Since \(|z(x_0)|=|y_M-y_P|\) can be made
	arbitrarily small by taking \(B\) sufficiently close to infinity, we may
	choose \(B\) so that the first term in \eqref{23} is smaller in absolute
	value than one half of this positive integral term. Hence \(z(x_*)>0\).
	Letting \(x_*\) approach \(\bar x\) from the left gives the desired
	comparison near \(\bar x\).
\end{rem}	
	
	\begin{rem}
		Here, we explain why the boundedness of any orbit sufficiently close to infinity cannot be determined when
		$
		F(x)\equiv F(-x)$
		for
		$x>\max\{|x_1|,|x_n|\}$,
		while
		$
		F(x)\not\equiv F(-x)$
		for
		$x>0$.
		Indeed, from \eqref{23}, one has
		\[
		z(x)
		=
		z(x_0)
		\exp\left(
		\int_{x_0}^{x}
		\phi(\xi)\,d\xi
		\right).
		\]
		In this case, the sign of \(z(x)\) cannot be determined.
	\end{rem}
\begin{proof}[Proof of Theorem~$\ref{thm0}$]
	Statements {\bf (i)--(iii)} follow from the classification of the
	equilibria at infinity given in Lemma~\ref{bounded:1}, whereas
	statement {\bf (iv)} follows from the boundedness properties of the
	orbits near infinity established in Lemma~\ref{bounded}.
\end{proof}

\section{The lower bound for the maximum number of limit cycles of the piecewise linear system}
\label{sc}
In this section,  we study  the  lower bounds for the maximum number of limit cycles of piecewise linear system   \eqref{ls} with $F(x)$
having $m$ jump points and $n-m$ fold points on $n+1$ intervals, where
$0\le m\le n$.
 To study the limit cycles of system \eqref{ls} with the piecewise linear function $F(x)$, we  first characterize the crossing and sliding sets on the switching line $L_i$. If the point $(x_i,F(x_i))$ is a fold point, then the crossing set on $L_i$ is given by
 	 $$
 	L_i^c=\{(x,y)\in L_i: y\ne F(x_i)\}.
 	$$
 If the point $(x_i,F(x_i))$ is a jump point, then the crossing and sliding sets on $L_i$ are defined by
 	$$
 	L_i^c=\left\{(x,y)\in L_i: y<\min\left\{F(x_i^-),F(x_i^+)\right\} \, \text{ or }\,  y>\max\left\{F(x_i^-),F(x_i^+)\right\} \right\}
 	$$
 	and
 		$$
 	L_i^s=\left\{(x,y)\in L_i: \min\left\{F(x_i^-),F(x_i^+)\right\}<y< \max\left\{ F(x_i^-),F(x_i^+)\right\} \right\},
 	$$
 	respectively.
 We further classify limit cycles as crossing, sliding, grazing and
 composite limit cycles.  For system \eqref{ls}, we have $\dot y = x$. Then
the orientation of the orbit along the sliding segment on the switching line $x=x_i$ is determined by $\sgn(x_i)$.

\begin{enumerate}[$(a)$]
\item
A periodic orbit $\Gamma$ is called a {\it crossing limit cycle} if it crosses
the switching line {transversally} {and its suitable neighborhood does not have other periodic orbits} $($see {\rm Figure~\ref{dia:21}(a)}$)$.

\item
A  periodic orbit $\Gamma$ is  a {\it sliding limit cycle} if { it contains
nonempty open segments $\Gamma_s$} in the sliding set and {$\Gamma\setminus\Gamma_s$ has a neighborhood  which does not contain orbit arcs outside the sliding set of any other sliding periodic orbit} $($see {\rm Figure~\ref{dia:21}(b,c)}$)$.

\item
A  periodic orbit $\Gamma$ is  a {\it grazing limit cycle} if it is  tangent
to the switching line at some point {  and it has a neighborhood which does not contain orbit arcs outside the sliding set of any other periodic orbit} $($see {\rm Figure~\ref{dia:21}(d,e)}$)$.

\item A periodic orbit \(\Gamma\) is a {\it composite limit cycle} if
at least two of the three behaviors in $(a)$, $(b)$ and $(c)$, namely
crossing, grazing and sliding, occur simultaneously on \(\Gamma\),
and if \(\Gamma\) has a neighborhood which does not contain orbit arcs
outside the sliding set of any other periodic orbit.
\end{enumerate}

\begin{figure}[htp]
	\centering
    	\subfloat[Crossing limit cycle $\Gamma$]
	{\includegraphics[scale=0.29]{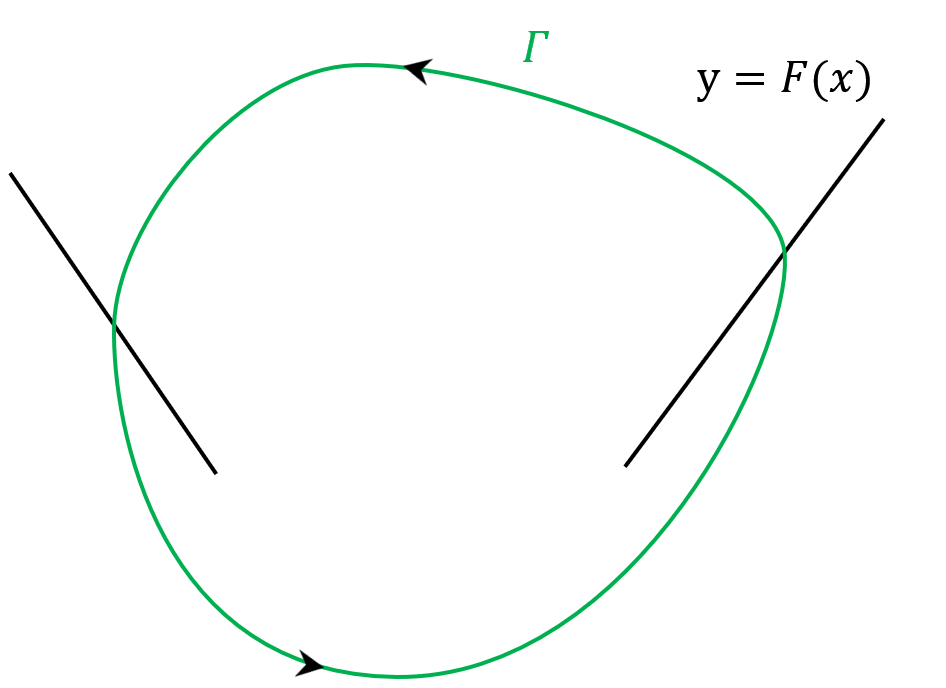}}
    	\subfloat[Sliding limit cycle $\Gamma$]
    {\includegraphics[scale=0.29]{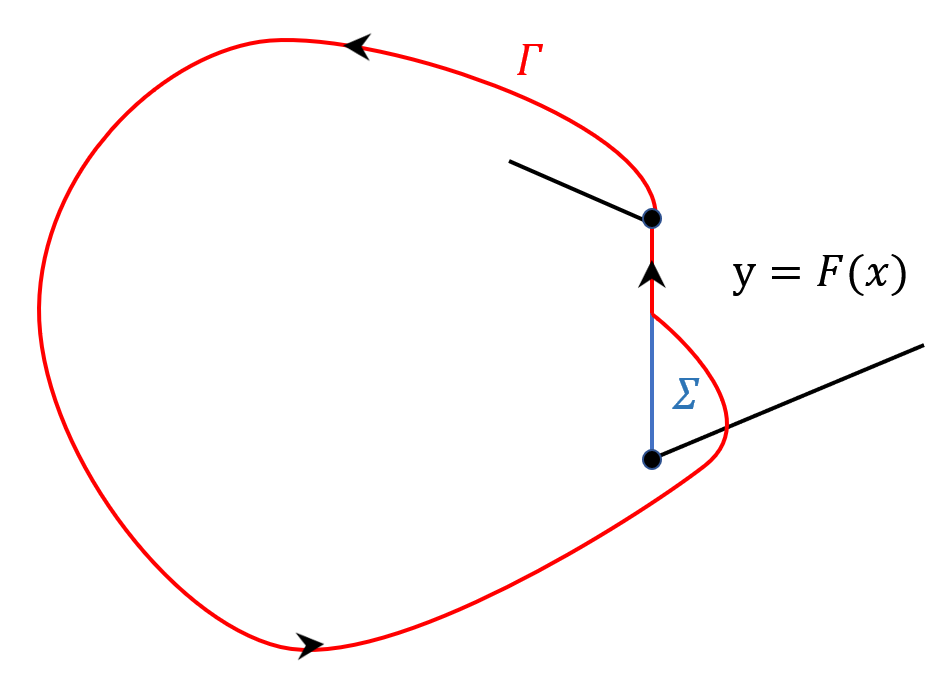}}
        	\subfloat[Sliding limit cycle $\Gamma$]
    {\includegraphics[scale=0.29]{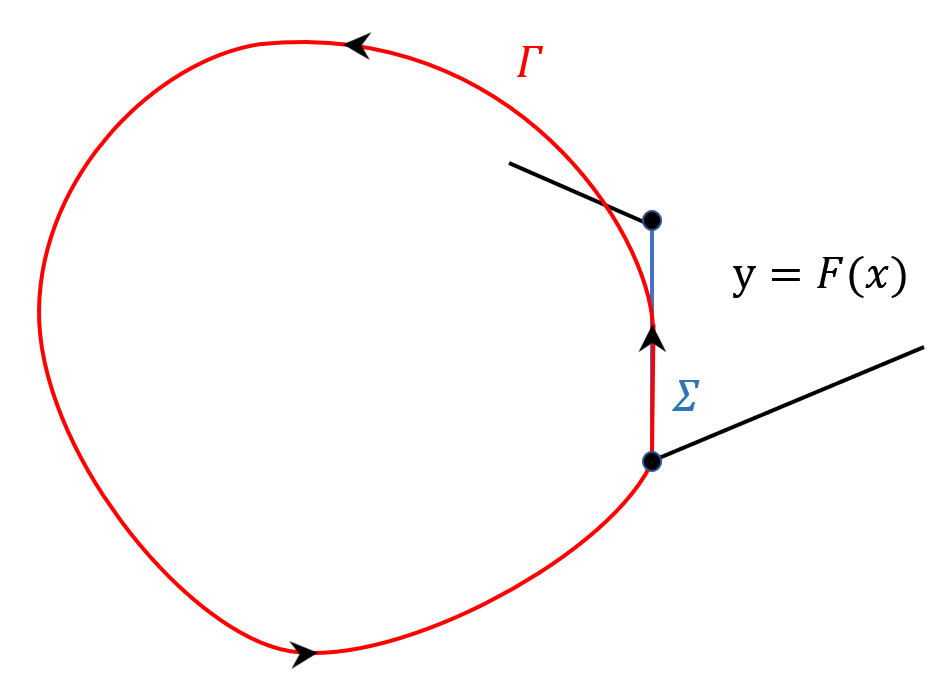}}\\
    	\subfloat[Grazing limit cycle $\Gamma$]
	{\includegraphics[scale=0.29]{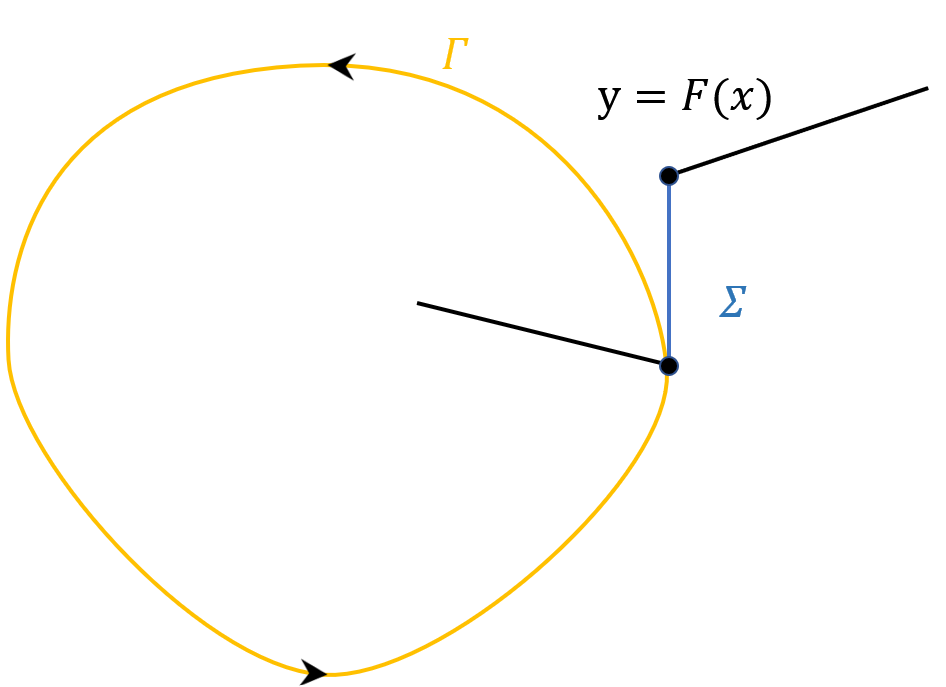}}
        	\subfloat[Grazing limit cycle $\Gamma$ ]
	{\includegraphics[scale=0.29]{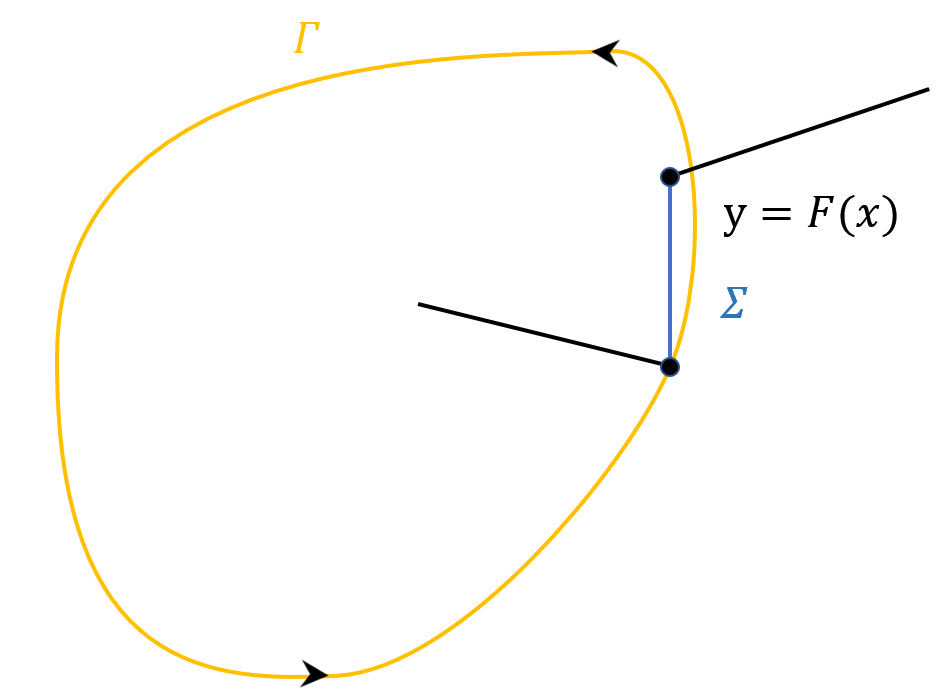}}
	\caption{Classification of limit cycles of system \eqref{ls} when $F(x)$ is piecewise linear ($\Sigma$ is the switching line).}
	\label{dia:21}
\end{figure}
With the above definitions in hand, we can now state the first result, which concerns part $(i)$ of Tonnelier's conjecture.
\begin{thm}\label{thm1}
	For $m=0$ and $n\ge 1$, a lower bound for the maximum number of limit cycles of system~\eqref{ls} is $n$.
\end{thm}
The next result concerns part $(ii)$  of Tonnelier's conjecture.
\begin{thm}\label{thm2}
	For $m=n\ge 1$, a lower bound for the maximum number of limit cycles of system~\eqref{ls} is $2n$.
\end{thm}
\begin{rem}
When $m=0$, the function $F(x)$ is Lipschitz continuous and has no jump
discontinuities.
Hence, all points $(x_i,F(x_i))$ are fold points and no sliding set exists on any switching line $L_i$ for $i=1,2,\cdots,n$. It follows that every limit cycle in Theorem~$\ref{thm1}$  is either a crossing one or a grazing one.

When $m=n$,  all points $(x_i,F(x_i))$ are jump points, and a sliding set exists on every switching line $L_i$ for $i=1,2,\cdots,n$. Hence, any limit cycle in Theorem~$\ref{thm2}$ may be a crossing, grazing, sliding, or composite limit cycle.
\end{rem}

These two theorems confirm the lower bounds in the
parts~(i) and~(ii) of Tonnelier's conjecture, namely, the exact numbers as predicted there.
We now extend these results to the mixed case involving both jump points and
fold points corresponding to Conjecture \ref{con:m}.

\begin{thm}\label{thm3}
For $0\le m\le n$, a lower bound for the maximum number of limit cycles of system~\eqref{ls} is $n+m$.
\end{thm}

\begin{rem}
Theorems~$\ref{thm1}$ and~$\ref{thm2}$ can be recovered as special cases of
Theorem~$\ref{thm3}$, {which confirms the number in Conjecture $\ref{con:m}$ as a lower bound}.
\end{rem}

\subsection{Proof of Theorem \ref{thm1}}
\label{sec2}

In this section, we construct a continuous piecewise linear system \eqref{ls} for which \(F(x)\), of the form \eqref{Fx-representation}, has exactly \(n\) fold points and the system admits at least \(n\) limit cycles, thereby proving Theorem~\ref{thm1}. Moreover, the constructed system satisfies the following properties:
\begin{enumerate}
	\item[\bf (I)] the switching lines \(x=x_i\) satisfy \(x_i>0\) for all \(i=1,2,\cdots,n\);
	
	\item[\bf (II)] \((a_l,a_r)\) belongs to case {\bf (3c)} and \(a_l<0\). Furthermore, \(a_l+a_r>0\) when \(n\) is odd, whereas \(a_l+a_r<0\) when \(n\) is even, where \(a_l\) and \(a_r\) denote the outermost slopes of \(F(x)\), given by \eqref{Fl} and \eqref{Fr}, respectively;
	
	\item[\bf (III)] the outermost limit cycle intersects all \(n+1\) intervals.
\end{enumerate}

	First, consider the case \(n=1\). In \cite{LPV}, Llibre, Ponce and Valls proved that system \eqref{ls} with
	\[
	F(x)=
	\begin{cases}
		a_l x, & \text{if } x\le x_1,\\
		a_r(x-x_1)+a_lx_1, & \text{if } x>x_1
	\end{cases}
	\]
	can admit one limit cycle, where \((a_l,a_r,x_1)\in\mathbb{R}^3\). Hence, Theorem~\ref{thm1} holds for \(n=1\).
	Next, consider the case \(n=2\). According to \cite{CFJM,CJT,CLXY,LPV} and the references therein, there exists a system \eqref{ls}, with \(F(x)\) piecewise linear on three intervals and having two fold points, that admits exactly two limit cycles. Hence, Theorem~\ref{thm1} also holds for \(n=2\).
	Moreover, by \cite[Theorem 2.4]{CFJM}, there exists a system~\eqref{ls} satisfying properties {\bf (I)--(III)} in which
	\(F(x)\) has two fold points and which possesses two limit cycles. 
	
	Assume inductively that, for every \(n<2k+1\) (\(k\ge 1\)), the system
	\begin{equation}
		\dot{x}=F_1^{(n)}(x)-y,
		\qquad
		\dot{y}=x
		\label{ls1}
	\end{equation}
	satisfies properties {\bf (I)--(III)} and admits at least \(n\) limit cycles for some parameters \(a^{(n)}\), \(b^{(n)}\) and \(a_i^{(n)}\neq 0\) (\(i=1,\dots,n\)), where
	\[
	F_1^{(n)}(x)
	=
	a^{(n)}x+b^{(n)}
	+
	\sum_{i=1}^{n}
	a_i^{(n)}
	|x-x_i|
	\]
	with \(x_1<\cdots<x_n\).
	Without loss of generality, denote by $
	\Gamma_1,\Gamma_2,\cdots,\Gamma_n$
	the outermost \(n\) limit cycles satisfying
	$
	\Gamma_1\subset\Gamma_2\subset\cdots\subset\Gamma_n
$.
	Moreover, \(\Gamma_n\) is externally unstable (resp. externally stable) when \(n\) is odd (resp. even). 	For convenience, denote the outermost slopes $a_l$ and $a_r$ of system \eqref{ls1} by
	$a_l^{(n)}$
	and
	$a_r^{(n)}$, respectively. Indeed, since \(\left(a_l^{(n)},a_r^{(n)}\right)\) belongs to case {\bf (3c)} and \(a_l^{(n)}+a_r^{(n)}\neq 0\), infinity is a focus by Theorem~\ref{thm0}. Since \(\Gamma_n\) is the outermost limit cycle, its external stability is opposite to the stability at infinity. It follows from Lemma~\ref{bounded} that \(\Gamma_n\) is externally unstable (resp. externally stable) when \(n\) is odd (resp. even). For $n=2k$, denote by \(x_*\) the abscissa of the rightmost point on \(\Gamma_{2k}\).
	
	We next show that, for \(n=2k+1\) and \(n=2k+2\), there exists a piecewise linear system with \(2k+1\) (resp. \(2k+2\)) fold points satisfying properties {\bf (I)--(III)} and admitting at least \(2k+1\) (resp. \(2k+2\)) limit cycles. 

	Consider the new system
	\begin{equation}
		\dot{x}=\bar F_1(x)-y,
		\qquad
		\dot{y}=x,
		\label{ls12}
	\end{equation}
	where
	\[
	\bar F_1(x)
	=
	a^{(2k)}x+b^{(2k)}
	+
	\sum_{i=1}^{2k}
	a_i^{(2k)}
	|x-x_i|
	+
	c_{2k+1}|x-x_{2k+1}|
	-
	c_{2k+1}(x_{2k+1}-x)
	\]
	with \(c_{2k+1}\neq 0\), \(x_*<x_{2k+1}\), and the remaining parameters coinciding with those in \(F_1^{(2k)}(x)\) of system \eqref{ls1}. 
	It is clear that system \eqref{ls12} can be rewritten in the form of \eqref{ls1} with \(n=2k+1\), and hence satisfies property {\bf (I)}. Observe that
	\[
	\bar F_1(x)
	=
	F_1^{(2k)}(x)
	\qquad
	\text{for }
	x\le x_{2k+1}.
	\]
	Therefore, for system \eqref{ls12}, the previously existing \(2k\) limit cycles $
	\Gamma_1,\Gamma_2,\cdots,\Gamma_{2k}$
	remain in the strip \(x\le x_{2k+1}\), and \(\Gamma_{2k}\) remains externally stable.
	
	To complete the induction step, it remains to show that for suitably chosen $c_{2k+1}$ and $x_{2k+1}$, system \eqref{ls12} admits an additional limit cycle, denoted by \(\Gamma_{2k+1}\). To this end, we apply the Poincar\'e--Bendixson annular region theorem. We first construct the inner boundary of the annular region.

	\medskip
	\noindent \textbf{Construction of the inner boundary:}
	Consider the orbit starting from
	$P_{2k+1}
	(x_{2k+1},
	\bar F_1(x_{2k+1}))$.
	Since \(\Gamma_{2k}\) is externally stable and there are no other limit cycles surrounding \(\Gamma_{2k}\) in the region \(x<x_{2k+1}\), the orbit \(\xi\) returns to the curve \(y=\bar F_1(x)\) at a point \(Q_{2k+1}\) in forward time, where \(Q_{2k+1}\) lies between \(P_{2k+1}\) and \(\Gamma_{2k}\), see Figure~\ref{dia:PB1}.
	\begin{figure}[!htbp]
		\centering
		\includegraphics[scale=0.5]{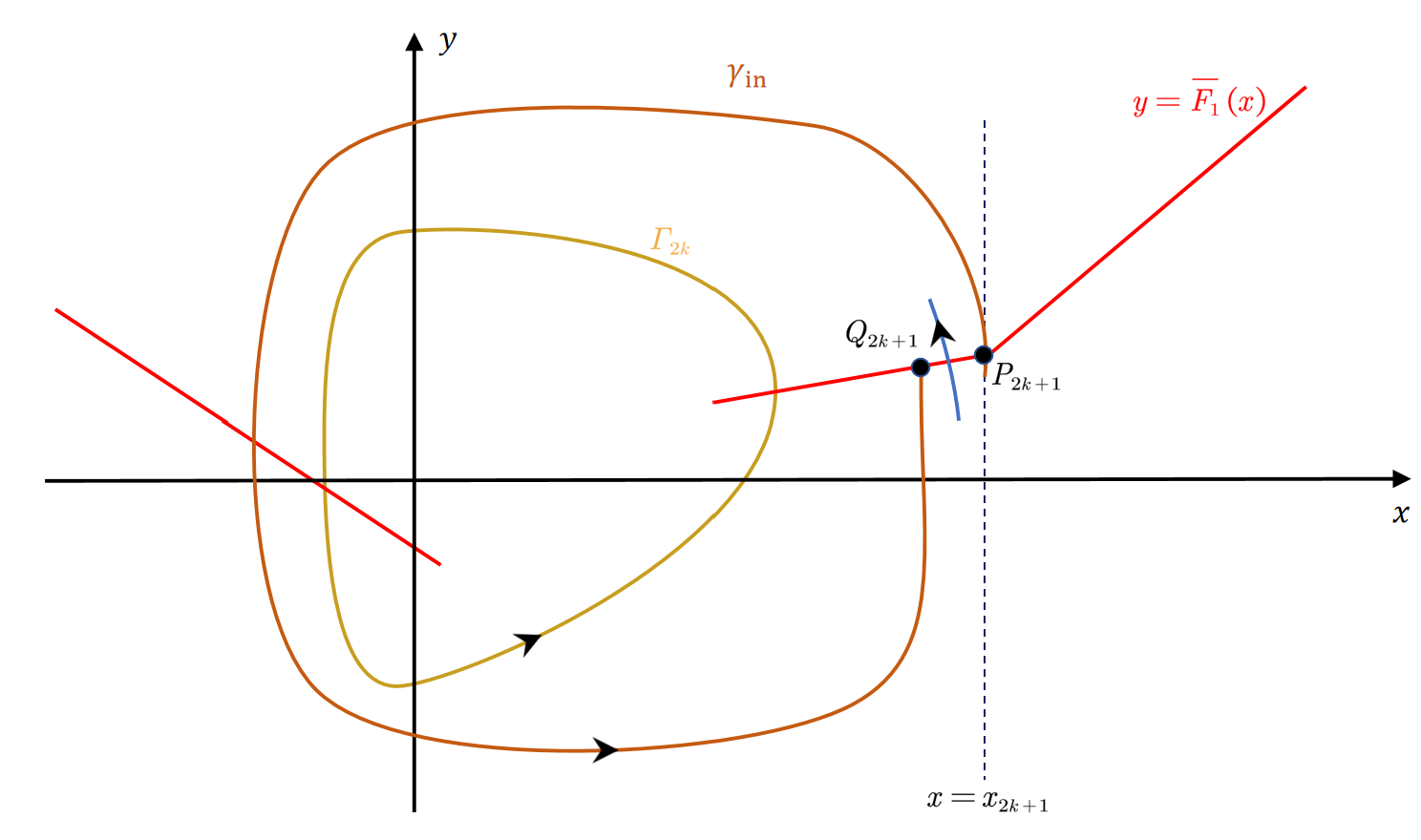}
		\caption{The inner boundary of the desired Poincar\'e--Bendixson annular region of system \eqref{ls12}.}
		\label{dia:PB1}
	\end{figure}

	Let \(\gamma_{\mathrm{in}}\) be the closed curve consisting of the orbit arc \(\widehat{P_{2k+1}Q_{2k+1}}\) and the line \(\overline{P_{2k+1}Q_{2k+1}}\). It remains only to check the direction of the vector field on the
	segment \(\overline{P_{2k+1}Q_{2k+1}}\). Since this segment is contained
	in the graph \(y=\bar F_1(x)\) and lies in the half-plane \(x>0\), we have
	\[
	\left.
	\frac{d}{dt}\bigl(y-\bar F_1(x)\bigr)
	\right|_{y=\bar F_1(x)}
	=
	\dot y-\bar F_1'(x)\dot x
	=
	x>0.
	\]
	Thus the vector field crosses the segment
	\(\overline{P_{2k+1}Q_{2k+1}}\) in one fixed direction. Hence \(\gamma_{\mathrm{in}}\) is a valid inner boundary of the Poincar\'e--Bendixson annular
	region.
	
	\medskip
	\noindent \textbf{Construction of the outer boundary:}
	When \(n=2k\), the induction hypothesis implies that \(\left(a_l^{(2k)},a_r^{(2k)}\right)\) belongs to case {\bf (3c)} and
	\[
	a_l^{(2k)}
	+
	a_r^{(2k)}
	=
	\left(
	a^{(2k)}
	-
	\sum_{i=1}^{2k}
	a_i^{(2k)}
	\right)
	+
	\left(
	a^{(2k)}
	+
	\sum_{i=1}^{2k}
	a_i^{(2k)}
	\right)
	<0,
	\qquad
	a_l^{(2k)}<0.
	\]
	Observe that the outermost slopes of system \eqref{ls12} are given by
	\[
	a_l
	=
	a^{(2k)}
	-
	\sum_{i=1}^{2k}
	a_i^{(2k)}
	=
	a_l^{(2k)}
	<0
	\]
	and
	\[
	a_r
	=
	a^{(2k)}
	+
	\sum_{i=1}^{2k}
	a_i^{(2k)}
	+
	2c_{2k+1}
	=
	a_r^{(2k)}
	+
	2c_{2k+1}.
	\]
	Since \(c_{2k+1}\in\mathbb{R}\) can be chosen arbitrarily, there always exists a choice of \(c_{2k+1}\) such that \((a_l,a_r)\) belongs to case {\bf (3c)} and 
	$a_l+a_r>0$.
	Hence, system \eqref{ls12} satisfies property {\bf (II)}.
Consequently, by Figure~\ref{dia:cen}(b) in Lemma~\ref{bounded}, the equator \(\mathbb S^1\) of system \eqref{ls12} serves as the outer boundary of the desired Poincar\'e--Bendixson annular region.

By the Poincar\'e--Bendixson theorem, there exists at least one limit cycle outside \(\gamma_{\mathrm{in}}\). We denote the outermost one by \(\Gamma_{2k+1}\). Clearly, it intersects all \(2k+2\) intervals, and is externally unstable. Thus, system \eqref{ls12} satisfies property {\bf (III)}.
Rewriting system \eqref{ls12} in the form of \eqref{ls1} with \(n=2k+1\), we conclude that system \eqref{ls1} satisfies properties {\bf (I)--(III)} and admits at least \(2k+1\) limit cycles.
 
For \(n=2k+2\), even though the number of equilibria at infinity depends on the properties of $a_l$ and $a_r$ (see  Lemma \ref{bounded:1}), by Lemma \ref{bounded} the stability of the infinity depends only on the sign of $a_l+a_r$. So, using similar arguments to those used in the proof of the case for \(n=2k+1\), one can construct a system \eqref{ls1} with \(n=2k+2\) satisfying properties {\bf (I)--(III)} and admitting at least \(2k+2\) limit cycles.

In summary, by induction one can construct a continuous piecewise system \eqref{ls} in which \(F(x)\) is  of the form \eqref{Fx-representation}, which has exactly \(n\) fold points and admits at least \(n\) limit cycles. This completes the proof of Theorem~\ref{thm1}.
\qed

\begin{rem}
Since $F(x)$ has only fold points,  the associated system \eqref{ls} has no sliding set.
Hence, all limit cycles constructed above are of crossing or grazing type, with no sliding dynamics.
This distinction becomes important in the proof of Theorem~$\ref{thm2}$.
\end{rem}

\subsection{Proof of Theorem \ref{thm2}}
\label{sec3}
To prove this theorem, we construct systems of the form~\eqref{ls} that admit at least \(2n\) limit cycles and in which \(F(x)\) is piecewise linear with exactly \(n\) jump points and no fold points. Moreover, the constructed system satisfies the following properties:
\begin{enumerate}
	\item[\bf (IV)] The switching lines \(x=x_i\) satisfy \(x_i>0\) for all \(i=1,2,\cdots,n\);
	
	\item[\bf (V)] \((a_l,a_r)\) belongs to case {\bf (3c)} and \(a_l+a_r<0\), \(a_l<0\),  where \(a_l\) and \(a_r\) denote the outermost slopes of \(F(x)\), given by \eqref{Fl} and \eqref{Fr}, respectively;
	
	\item[\bf (VI)]  The outermost limit cycle is a crossing one intersecting all \(n+1\) intervals, and there is a sliding limit cycle whose sliding segment lies on the switching line \(x=x_n\).
\end{enumerate}

Our procedure, which is similar to that used in the proof of
Theorem~\ref{thm1}, proceeds by induction on the number of jump points
of \(F(x)\). But due to the presence of sliding dynamics, we need to carefully handle the sliding dynamics in constructing the two distinct annular regions to obtain the two additional limit cycles.


First, consider the case \(n=1\). In \cite{CDFZ}, the authors  proved that the system
\begin{equation}\label{s}
\dot x =y-F(x),\qquad \dot y=1-x
\end{equation}
with
\[
F(x)=
\begin{cases}
t_lx+\beta, & \text{if } x<0,\\
t_rx-\beta, & \text{if } x>0
\end{cases}
\]
has at most two limit cycles, where \((t_l,t_r,\beta)\in\mathbb{R}^2\times \mathbb R^+\). Moreover, \cite[Theorem~2]{CDFZ} provides an example satisfying
properties {\bf (IV)--(VI)} in which \(F(x)\) has exactly one jump
point and no fold points, and the corresponding system possesses
exactly two limit cycles. Since system \eqref{s} can be transformed into system \eqref{ls} by $(1-x,y)\mapsto (x,y)$, Theorem~\ref{thm2} holds for \(n=1\).

	Assume inductively that, for every  $n<k$ (\(k\ge 2\)), the system
\begin{equation}
	\dot{x}=F_2^{(n)}(x)-y,
	\qquad
	\dot{y}=x
	\label{ls2}
\end{equation}
satisfies properties {\bf (IV)--(VI)} and admits at least \(2n\) limit cycles for some parameters \(a^{(n)}\), \(b^{(n)}\), \(a_i^{(n)}\) and \(b_i^{(n)}\ne 0\) with  \(i=1,\dots,n\), where
$$
F_2^{(n)}=a^{(n)}x+b^{(n)}+\sum_{i=1}^{n}a^{(n)}_i|x-x_i|+\sum_{i=1}^nb^{(n)}_i\sgn(x-x_i)
$$
and \(x_1<\cdots<x_n\).
Without loss of generality, denote by $
\gamma_1,\gamma_2,\cdots,\gamma_{2n}$
the outermost \(2n\) limit cycles satisfying
$
\gamma_1\subset\gamma_2\subset\cdots\subset\gamma_{2n}
$.
Moreover, it follows from Figure \ref{dia:cen}(a) that  \(\gamma_{2n}\) is externally stable. Let $x_*^{(n)}$ be  the abscissa of the rightmost point of $\gamma_{2n}$. Then $x_n<x_*^{(n)}$.

Consider the system
\begin{equation}
\dot{x}=\bar F_2(x)-y,\qquad \dot{y}=x,
\label{ls22}
\end{equation}
where
\begin{equation}\label{2}
\begin{aligned}
\bar F_2(x)=&
a^{(k-1)}x+b^{(k-1)}+\sum_{i=1}^{k-1}a_i^{(k-1)}|x-x_i|+\sum_{i=1}^{k-1}b_i^{(k-1)}\sgn(x-x_i)\\
&\qquad\qquad\quad +c_k|x-x_k|-c_k(x_{k}-x)+d_{k}\sgn(x-x_k)+d_k
\end{aligned}
\end{equation}
with $x_{k}>x_*^{(k-1)}$,  $c_k, d_k\in\mathbb R$, and all other parameters taken from $F_2^{(k-1)}(x)$.
It is clear that we can rewrite \eqref{ls22} in the form of \eqref{ls2} with $n=k$, satisfying property {\bf (IV)}.
 Moreover, one can check that
$F_2^{(k-1)}(x)=\bar F_2(x)$ on the strip $x\leq x_k$, and so   systems \eqref{ls2} and \eqref{ls22} are the same in the region $x\leq x_{k}$, which contains the limit cycles $\gamma_1,\gamma_2,\dots,\gamma_{2k-2}$. Consequently,  $\gamma_1,\gamma_2,\dots,\gamma_{2k-2}$ are  also limit cycles of system \eqref{ls22} for those values of $a^{(k-1)}$, $b^{(k-1)}$, $a_1^{(k-1)},b_1^{(k-1)},\dots,a_{k-1}^{(k-1)},b^{(k-1)}_{k-1}$ in system \eqref{ls22}. Next, to complete the inductive argument, we will prove that
 for suitably chosen $c_{k}$, $d_k$ and $x_{k}$, system \eqref{ls22} admits two additional  limit cycles,  denoted by \(\gamma_{2k-1}\) and $\gamma_{2k}$ with $\gamma_{2k-1}$ a sliding cycle and $\gamma_{2k}$ a crossing cycle.

 \medskip
 \noindent \textbf{Construction of the first additional  sliding limit cycle $\gamma_{2k-1}$ outside $\gamma_{2k-2}$:} We return to system \eqref{ls2} with \(n=k-1\). Choose a point
 $P_{k-1}
 \left(
 x_{P_{k-1}},
 F_2^{(k-1)}(x_{P_{k-1}})
 \right)$
 on the curve \(y=F_2^{(k-1)}(x)\) satisfying
 $x_{P_{k-1}}>x_*^{(k-1)}$,
 and let \(\zeta_{k-1}\) be the orbit arc starting from \(P_{k-1}\).
 \begin{figure}[!htbp]
 	\centering
 	\includegraphics[scale=0.5]{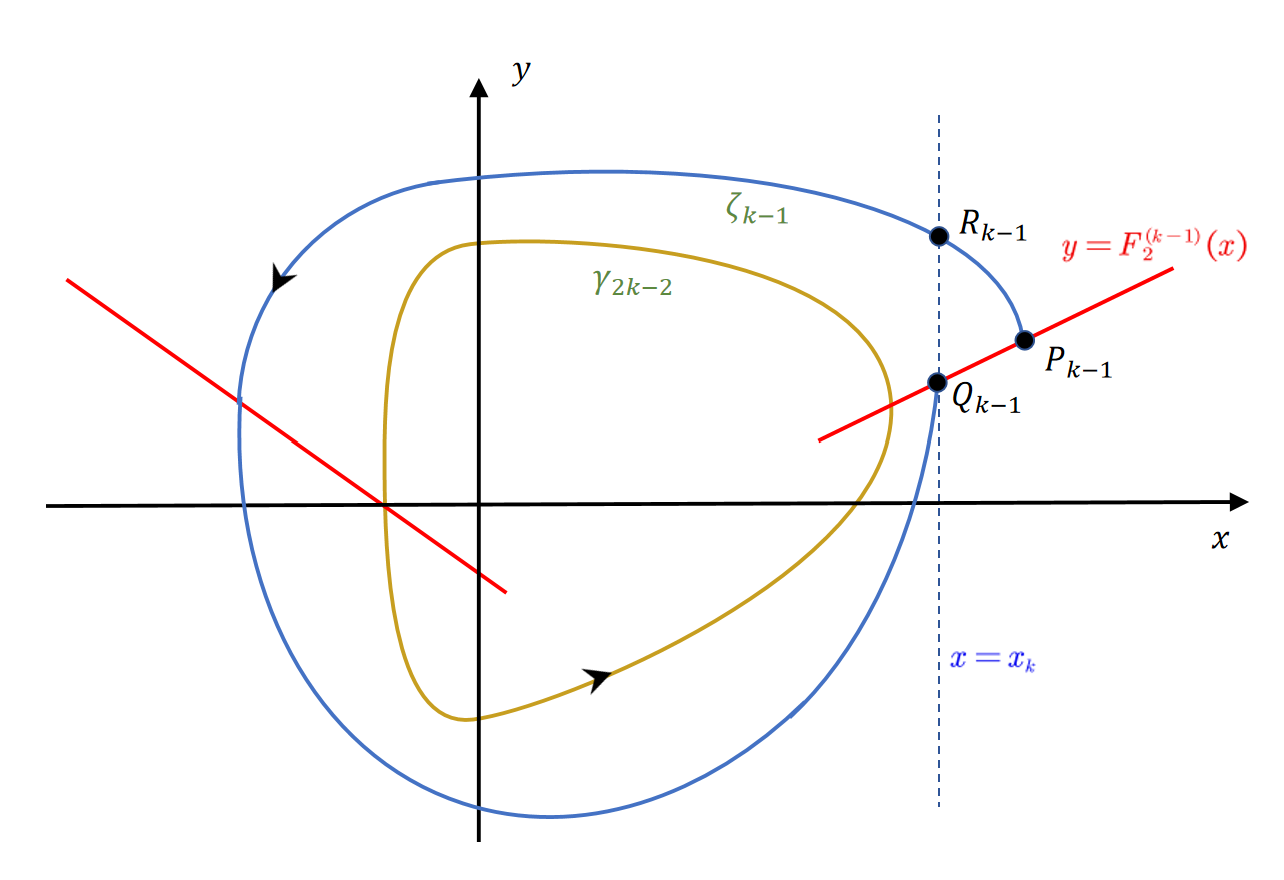}
 	\caption{The outermost limit cycle $\gamma_{2k-2}$ of system \eqref{ls2} with $n=k-1$.}
 	\label{dia:PB2}
 \end{figure}
 As illustrated in Figure~\ref{dia:PB2}, since \(\gamma_{2k-2}\) is the outermost crossing limit cycle and is externally stable, all intersection points of \(\zeta_{k-1}\) with the switching lines \(L_i\) lie on \(L_i^c\). Moreover, \(\zeta_{k-1}\) intersects the curve \(y=F_2^{(k-1)}(x)\) again. Denote by
 $Q_{k-1}
 \left(
 x_{Q_{k-1}},
 F_2^{(k-1)}(x_{Q_{k-1}})
 \right)$
 the first intersection point on \(y=F_2^{(k-1)}(x)\). Consequently,
 $x_*^{(k-1)}
 <
 x_{Q_{k-1}}
 <
 x_{P_{k-1}}$.
 Let \(R_{k-1}\) be the point on \(\zeta_{k-1}\) having the same \(x\)-coordinate as \(Q_{k-1}\). In view of the above analysis, for the system  of the form \eqref{ls22} to be constructed, we may set
 $x_k=x_{Q_{k-1}}$.

We now determine \(d_k\). For system \eqref{ls22}, it follows from \eqref{2} that
\[
\lim_{x\to x_k^-}\bar F_2(x)
=
a^{(k-1)}x_k+b^{(k-1)}
+
\sum_{i=1}^{k-1}
a_i^{(k-1)}(x_k-x_i)
+
\sum_{i=1}^{k-1}
b_i^{(k-1)}
\]
and
\[
\lim_{x\to x_k^+}\bar F_2(x)
=
a^{(k-1)}x_k+b^{(k-1)}
+
\sum_{i=1}^{k-1}
a_i^{(k-1)}(x_k-x_i)
+
\sum_{i=1}^{k-1}
b_i^{(k-1)}
+
2d_k .
\]

By choosing \(P_{k-1}\) sufficiently close to the outer side of
\(\gamma_{2k-2}\), and using the continuous dependence of solutions on
initial points, we may assume that the point \(R_{k-1}\), which has the
same abscissa as \(Q_{k-1}\), satisfies
\[
y_{R_{k-1}}>y_{Q_{k-1}} .
\]
It is clear that
\[
\lim_{x\to x_k^-}\bar F_2(x)
\]
coincides with the ordinate of \(Q_{k-1}\). Hence one can choose
\[
d_k=\frac{y_{R_{k-1}}-y_{Q_{k-1}}}{2}>0
\]
such that
\[
\lim_{x\to x_k^+}\bar F_2(x)
\]
coincides with the ordinate of \(R_{k-1}\).

Moreover, for every point \((x_k,y)\in \overline{Q_{k-1}R_{k-1}}\), we have
\[
\lim_{x\to x_k^-}\bar F_2(x)<y<
\lim_{x\to x_k^+}\bar F_2(x).
\]
Therefore,
\[
\bar F_2(x_k^-)-y<0,
\qquad
\bar F_2(x_k^+)-y>0.
\]
Thus the two vector fields point to opposite sides of the switching line
\(x=x_k\), and \(\overline{Q_{k-1}R_{k-1}}\) is a Filippov sliding segment.

Since systems \eqref{ls2} and \eqref{ls22} coincide in the region
\(x\le x_k\), the orbit arc \(\widehat{R_{k-1}Q_{k-1}}\) of system
\eqref{ls2} is also an orbit arc of system \eqref{ls22}. Therefore, in
system \eqref{ls22}, the orbit arc \(\widehat{R_{k-1}Q_{k-1}}\), together
with the sliding segment \(\overline{Q_{k-1}R_{k-1}}\), forms a sliding
limit cycle, denoted by \(\gamma_{2k-1}\), see Figure~\ref{dia:PB3}.
 \begin{figure}[!htbp]
	\centering
	\includegraphics[scale=0.5]{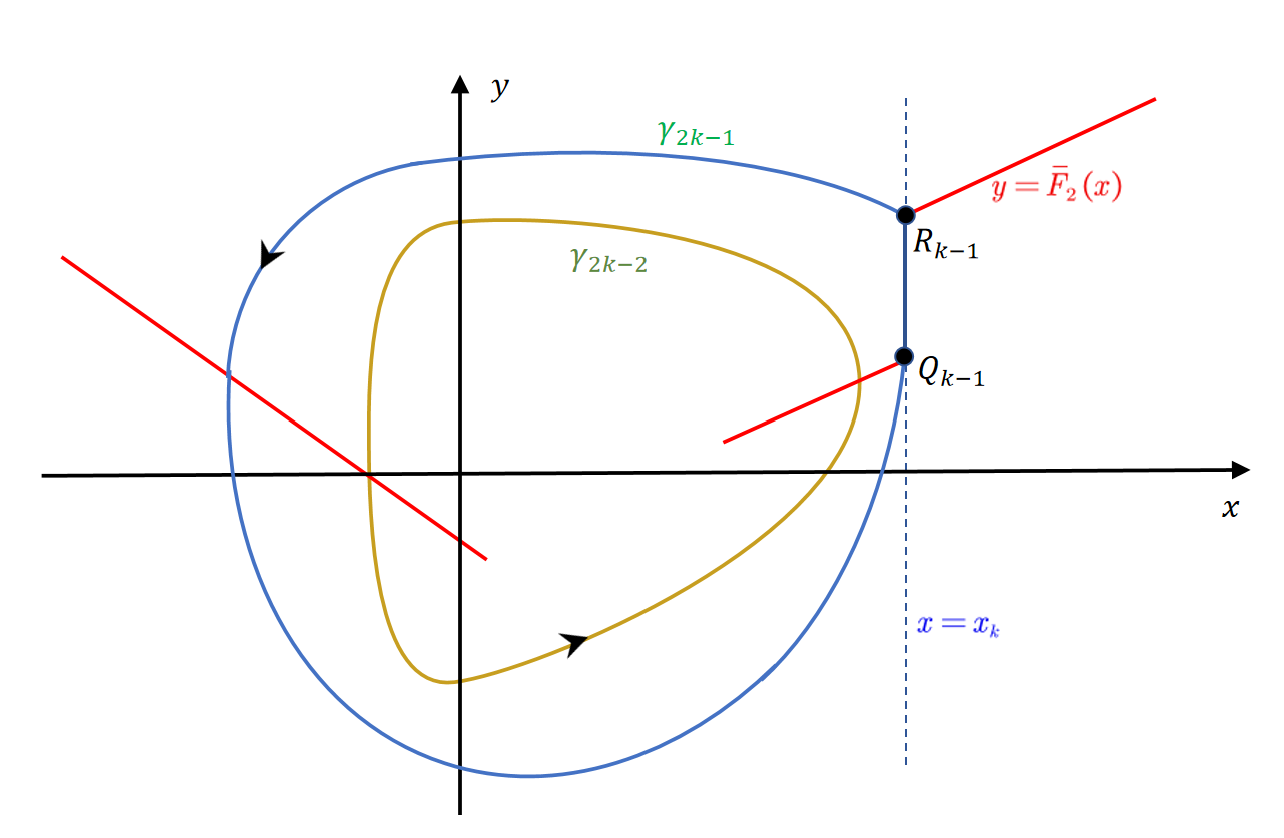}
	\caption{The sliding limit cycle $\gamma_{2k-1}$ of system \eqref{ls22}.}
	\label{dia:PB3}
\end{figure}

 \medskip
 \noindent \textbf{Construction of the second additional limit cycle \(\gamma_{2k}\) outside \(\gamma_{2k-1}\):}
 We first recall the stability of the sliding limit cycle \(\gamma_{2k-1}\).
 On the left-hand and right-hand sides of the switching line
 $L_k:\ x=x_k=x_{Q_{k-1}}$,
 the vector fields associated with system \eqref{ls22} are given, respectively, by
 \[
 \begin{split}
 	X^-(x,y)
 	&=
 	\big(\bar F^-_2(x)-y,\;x\big),
 	\qquad x<x_k,
 	\\
 	X^+(x,y)
 	&=
 	\big(\bar F^+_2(x)-y,\;x\big),
 	\qquad x>x_k.
 \end{split}
 \]
 If along the sliding segment \(\overline{R_{k-1}Q_{k-1}}\),
 \[
 \langle \nabla (x-x_k), X^+(x,y) \rangle < 0
 \quad \text{and} \quad
 \langle \nabla (x-x_k), X^-(x,y) \rangle > 0
 \]
 (resp.
 \[
 \langle \nabla (x-x_k), X^+(x,y) \rangle > 0
 \quad \text{and} \quad
 \langle \nabla (x-x_k), X^-(x,y) \rangle < 0),
 \]
 then the sliding limit cycle is stable (resp. unstable), where \(\nabla\) denotes the gradient operator and \(\langle \cdot,\cdot\rangle\) the inner product.
 By direct computation,
 \[
 \left.
 \langle \nabla (x-x_k), X^\pm(x,y)\rangle
 \right|_{x=x_k}
 =
 \bar F_2^\pm(x_k)-y,
 \qquad \mbox{\rm on }L_k.
 \]
 Hence, the sliding limit cycle \(\gamma_{2k-1}\) is unstable.

 We next determine a suitable \(c_k\) such that system \eqref{ls22} admits a crossing limit cycle \(\gamma_{2k}\) outside \(\gamma_{2k-1}\).
 By the induction hypothesis, for system \eqref{ls2} with \(n=k-1\),
 \[
 a_l^{(k-1)}
 =
 a^{(k-1)}
 -
 \sum_{i=1}^{k-1}
 a_i^{(k-1)}
 <0
 \]
 and
 \[
 a_r^{(k-1)}
 =
 a^{(k-1)}
 +
 \sum_{i=1}^{k-1}
 a_i^{(k-1)},
 \qquad
 a_l^{(k-1)}+a_r^{(k-1)}<0,
 \]
 where \(a_l^{(k-1)}\) and \(a_r^{(k-1)}\) denote the outermost slopes of system \eqref{ls2}.
 For system \eqref{ls22}, the outermost slopes satisfy
 \[
 a_l
 =
 a^{(k-1)}
 -
 \sum_{i=1}^{k-1}
 a_i^{(k-1)}
 =
 a_l^{(k-1)}
 <0
 \]
 and
 \[
 a_r
 =
 a^{(k-1)}
 +
 \sum_{i=1}^{k-1}
 a_i^{(k-1)}
 +
 2c_k
 =
 a_r^{(k-1)}
 +
 2c_k.
 \]
 Since \(c_k\in\mathbb{R}\) can be chosen arbitrarily, there always exists a choice of \(c_k\) such that \((a_l,a_r)\) belongs to case {\bf (3c)} and
 $a_l+a_r<0$.
 In particular, one convenient choice is \(c_k=0\).
 Thus, system \eqref{ls22} satisfies property {\bf (V)}, implying that infinity of system \eqref{ls22} is an unstable focus, see Figure~\ref{dia:cen}(a). Thus, in the annular region between \(\gamma_{2k-1}\) and the equator
 \(\mathbb S^1\), the vector field points into the annulus along both
 boundary components. Moreover, this annular region contains no equilibrium.
 Therefore, by the Poincar\'e--Bendixson theorem, there exists at least one
 periodic orbit in this annulus. Since the sliding segment on the newly added switching line lies on the
 inner boundary \(\gamma_{2k-1}\), the obtained periodic orbit is a crossing
 limit cycle.
 In view of the above analysis, one can choose suitable values of \(c_k\), \(d_k\), and \(x_k\) such that system \eqref{ls22} satisfies properties {\bf (IV)--(VI)}, and consequently admits at least \(2k\) limit cycles.

In summary,  by induction it follows that when system \eqref{ls} has \(n\) jump points and no fold points, it can exhibit at least \(2n\) limit cycles. This completes the proof of Theorem~\ref{thm2}.
\qed

\subsection{Proof of Theorem \ref{thm3}}
\label{sec4}
In this section, we construct the systems of the form \eqref{ls}, with the piecewise linear function $F(x)$ having exactly $m$ jump points and $n-m$ fold points, that have at least $m+n$ limit cycles. 
The proof combines the constructions used in the proofs of
Theorems~\ref{thm1} and~\ref{thm2}. We construct systems of the
form~\eqref{ls} satisfying the following properties:
\begin{enumerate}
\item[\bf (i)] The switching lines \(x=x_i\) satisfy \(x_i>0\) for
all \(i=1,2,\ldots,n\). Moreover, \(S_i=(x_i,F(x_i))\) is a jump
point for \(i=1,\ldots,m\), whereas \(S_j=(x_j,F(x_j))\) is a fold
point for \(j=m+1,\ldots,n\);
	
	\item[\bf (ii)] \((a_l,a_r)\) belongs to case {\bf (3c)} and \(a_l<0\). Furthermore, \(a_l+a_r>0\) when \(n-m\) is odd, whereas \(a_l+a_r<0\) when \(n-m\) is even, where \(a_l\) and \(a_r\) denote the outermost slopes of \(F(x)\), given by \eqref{Fl} and \eqref{Fr}, respectively;
	
	\item[\bf (iii)] The outermost limit cycle intersects all \(n+1\) intervals.
\end{enumerate}
Based on Theorem~\ref{thm2}, for any given \(m\), one can find a system of the form
\begin{equation}
	\dot{x}=F^{(m)}(x)-y,
	\qquad
	\dot{y}=x
	\label{lsm}
\end{equation}
satisfying properties {\bf (IV)--(VI)} and admitting at least \(2m\) limit cycles, where
\[
F^{(m)}(x)
=
a^{(m)}x+b^{(m)}
+
\sum_{i=1}^{m}
a_i^{(m)}
|x-x_i|
+
\sum_{i=1}^{m}
b_i^{(m)}
\sgn(x-x_i)
\]
with
$x_1<\cdots<x_m$ and $b_i^{(m)}\ne 0$.

Next, consider the system
\begin{equation}
	\dot{x}=\bar F_3(x)-y,
	\qquad
	\dot{y}=x,
	\label{l}
\end{equation}
where
\[
\bar F_3(x)
=
a^{(m)}x+b^{(m)}
+
\sum_{i=1}^{m}
a_i^{(m)}
|x-x_i|
+
\sum_{i=1}^{m}
b_i^{(m)}
\sgn(x-x_i)
+
c_{m+1}|x-x_{m+1}|
-
c_{m+1}(x_{m+1}-x)
\]
with \(c_{m+1}\in\mathbb{R}\), and the line \(x=x_{m+1}\) lying strictly to the right of the outermost limit cycle of system \eqref{lsm}.
It follows from the arguments in Subsection~\ref{sec2} that, by choosing a suitable \(c_{m+1}\), system \eqref{l} admits at least \(2m+1\) limit cycles and satisfies properties {\bf (i)--(iii)}. Therefore, Theorem~\ref{thm3} holds for \(n-m=1\).

Furthermore, starting from the system constructed in the proof of
Theorem~\ref{thm2}, we have a system with \(m\) jump points and at least
\(2m\) limit cycles. We now add the remaining \(n-m\) fold points one by
one.

Assume that, after adding \(j\) fold points, where \(0\leq j<n-m\), the
constructed system has \(m\) jump points, \(j\) fold points, and at least
\(2m+j\) limit cycles. Denote the outermost one by \(\Gamma_{2m+j}\).
Choose the next switching point \(x_{m+j+1}\) strictly to the right of
\(\Gamma_{2m+j}\). The modification of \(F\) is made only in the region
\(x>x_{m+j+1}\). Hence the modified system coincides with the previous
one in the whole region containing all previously constructed limit cycles.
Therefore all previously constructed limit cycles persist.

Moreover, the newly added switching point is a fold point. Since it is
placed outside the former outermost limit cycle, the same
Poincar\'e--Bendixson annular-region argument as in the proof of
Theorem~\ref{thm1} gives at least one additional limit cycle outside
\(\Gamma_{2m+j}\). Thus, after this step, the system has at least
\(2m+j+1\) limit cycles.

Repeating this procedure for \(j=0,1,\ldots,n-m-1\), we obtain a system
\eqref{ls} with
\[
F(x)
=
a^{(m)}x+b^{(m)}
+
\sum_{i=1}^{m}
a_i^{(m)}
|x-x_i|
+
\sum_{i=1}^{m}
b_i^{(m)}
\sgn(x-x_i)
+
\sum_{j=m+1}^{n}
\left(
c_j|x-x_j|
-
c_j(x_j-x)
\right),
\]
which satisfies properties {\bf (i)--(iii)} and admits at least
\[
2m+(n-m)=n+m
\]
limit cycles. Therefore, Theorem~\ref{thm3} holds.
\qed
\subsection{Concluding remarks}
\label{sec5}
Theorems~\ref{thm1}--\ref{thm3} provide lower bounds for the maximum number of limit cycles for piecewise linear Li\'enard differential systems. To establish these results, we construct piecewise linear systems satisfying certain prescribed properties, such as {\bf (I)--(III)}, {\bf (IV)--(VI)} and {\bf (i)--(iii)}. It should be emphasized that these properties are not necessary conditions for the existence of the corresponding number of limit cycles; rather, they are imposed solely to facilitate the construction of the desired systems.
For example, in the system constructed in Theorem~\ref{thm3}, the first $m$ switching lines induce $m$ jump points, whereas the remaining $n-m$ switching lines induce $n-m$ fold points. This assumption is rather restrictive. In fact, systems in which jump points and fold points occur in an arbitrary arrangement can also be constructed. Moreover, in the proof of the theorem, we assume that $(a_l,a_r)$ belongs to case {\bf (3c)} so that the dynamics at infinity is either stable or unstable, thereby enabling the construction of a Poincar\'e--Bendixson annular region.

Indeed, Theorem~\ref{thm3} may be viewed as a generalization of Theorems~\ref{thm1} and~\ref{thm2}. The proofs of Theorems~\ref{thm1} and~\ref{thm2} are based on the mathematical induction with respect to $n$. Suppose that for $n<k$, there exists a system possessing $n$ or $2n$ limit cycles. The inductive step is achieved by introducing new switching lines on the right-hand side of the outermost limit cycle of the original system, thereby generating additional fold points or jump points. This also explains why all switching lines in the constructed systems satisfy $x_i>0$.

In the case of a fold point, suitable parameters are selected to alter the stability at infinity. Then, by applying the Poincar\'e--Bendixson annular region theorem, an additional limit cycle can be constructed. In the case of a jump point, a sliding limit cycle can be obtained directly. Furthermore, the sliding limit cycle together with the equator ($\mathbb S^1$) forms a Poincar\'e--Bendixson annular region, which enables the construction of another limit cycle. In other words, a fold point may generate one additional limit cycle, whereas a jump point may produce two additional limit cycles.

\section*{Acknowledgments}

This research is partially supported by the National Key R\&D Program of China grant number 2022YFA1005900.

The first, second and third authors are supported by the  National Natural Science Foundation of China (Nos. 12322109, 12571190), Hunan Basic Science Research Center for Mathematical Analysis (2024JC2002).  The second author is supported by Hunan Provincial Innovation Foundation for Postgraduate (No. CX20250155). The third author is supported by the Postdoctoral Fellowship Program and China Postdoctoral Science Foundationunder Grant Number BX2026012. The fourth  author is partially supported by the  National Natural Science Foundation of China (No. 12471169).

	\end{document}